\documentclass[11pt]{amsart}
\usepackage[usenames, dvipsnames]{color}
\usepackage{amssymb}

\theoremstyle{remark}

\definecolor{titlecol}{named}{BrickRed}
\definecolor{headcol}{named}{Violet}
\definecolor{seccol}{named}{Red}
\definecolor{sseccol}{named}{Bittersweet}
\definecolor{pbcol}{named}{Black}
\definecolor{sncol}{named}{Brown}
\definecolor{acol1}{named}{Red}
\definecolor{acol2}{named}{Apricot}

\def\p{\partial}
\def\R{\mathbb{R}}

\def\l{\lambda}

\def\cC{{\mathcal C}}

\def\cL{{\mathcal L}}

\def\cR{{\mathcal R}}
\def\cS{{\mathcal S}}

\theoremstyle{plain}
\newtheorem{thm}{Theorem}
\newtheorem{prop}{Proposition}[section]
\newtheorem{lemma}[prop]{Lemma}
\newtheorem{cor}[prop]{Corollary}
\newtheorem{rmk}[prop]{Remark}

\newtheorem{defn}[prop]{Definition}

\numberwithin{prop}{section}

\begin{document}

\title{The Gursky-Streets equation and its application to the $\sigma_k$ Yamabe problem}
\author{Weiyong He}
\address{Department of Mathematics, University of Oregon, Eugene, OR 97403.
}
\email{whe@uoregon.edu}
\author{Lu Xu}
\address{
Institute of Mathematics,
 Hunan University, Changsha 410082 China.}
 \email{xulu@hnu.edu.cn}
 \author{ Mingbo Zhang}
 \address{Department of Mathematics, University of Science and Technology of China, Hefei 230026
China.}
 \email{mbzhang@ustc.edu.cn}
 \thanks{The first author is supported in part by  an NSF grant, No. 1611797. The second author is supported by NSFC No. 11771132 and Hunan Science and Technology Project No. 2018JJ1004. The third author is supported by Youth Innovation Fund 2016 under grant No. WK0010460002. }


\maketitle

\begin{center}
{Abstract}
\end{center}

The Gursky-Streets equation are introduced as the geodesic equation of a metric structure in conformal geometry. This geometric structure has played a substantial role in the proof of uniqueness of $\sigma_2$ Yamabe problem in dimension four.
In this paper we solve the Gursky-Streets equations with uniform $C^{1, 1}$ estimates for $2k\leq n$.
An important new ingredient is to show the concavity of the operator which holds for all $k\leq n$. Our proof of the concavity heavily relies on Garding's theory of hyperbolic polynomials and results from the theory of real roots for (interlacing) polynomials.
Together with this concavity, we are able to solve the equation with the uniform $C^{1, 1}$ \emph{a priori estimates} for all the cases $n\geq 2k$. Moreover, we establish the uniqueness of the solution to the degenerate equations for the first time.

As an application, we prove that if $k\geq 3$ and $M^{2k}$ is conformally flat,  any solution solution of $\sigma_k$ Yamabe problem is conformal diffeomorphic to the round sphere $S^{2k}$.







\vskip25pt
\section{Introduction}
\vskip5pt

Let $(M^n, g)$ be a compact Riemannian manifold of dimension $n$, $n\geq 3$, with a fixed conformal class $[g]$. Write $Ric$ for the Ricci tensor of $g$. The Schouten tensor  is defined as
\[
A:=\frac{1}{n-2}\left(Ric-\frac{1}{2(n-1)} Rg\right).
\]
The well-known \emph{$\sigma_k$-curvature} is the $k$-th elementary symmetric function of the eigenvalues of $g^{-1}A$. For $1\leq k\leq n$, we say $A\in\Gamma^+_k$ if $\sigma_j(g^{-1}A)>0$ for all $1\leq j\leq k$. Suppose $(M^n, g)$ satisfies that $A\in \Gamma^+_k$.
Let $g_u=e^{-2u}g$ be a conformal metric. We say $u$ is \emph{admissible}, if $A_u\in \Gamma^+_k$, where $A_u$ is the Schouton tensor of $g_u$.  Denote
\[
\cC^+_k=\left\{u| g_u\in [g], A_u\in \Gamma^+_k\right\}.
\]
Given $u_0, u_1\in \cC^+_k$, we study the following equations.
\begin{equation}\label{GS1}
u_{tt}\sigma_k(A_u)-\langle T_{k-1}(A_u),  \nabla u_t\otimes \nabla u_t\rangle=0.
\end{equation}
This equation is introduced by Gursky and Streets \cite{GS-2017}. When $k=2, n=4$, Gursky-Streets \cite{GS-2017} studied the  metric structure extensively  on the space of $\cC^{+}_2$  defined by
\begin{equation}\label{gsmetric}
\langle \psi, \phi\rangle_{u}=\int_M \phi \psi \sigma_2(g_u^{-1}A_u) dV_u.
\end{equation}
and explored its geometry. This geometric structure plays a substantial role in the proof of the uniqueness of $\sigma_2$-Yamabe problem on a compact four manifold.

In this article,  we aim to further study the general case of the Gursky-Streets equation. When $M$ is conformally flat and $n=2k\geq 6$, the geometric structure is parallel to the case $n=4=2k$, see \cite{GS-2017} and \cite{GS-2018}.

It is noteworthy that the geometry of Gursky-Streets' metric on the space of conformal metrics has a parallel theory with the geometry of the space of K\"ahler metrics. It was originally inspired by the Mabuchi-Semmes-Donaldson metric \cite{M1, M2, Semmes, D04} of K\"ahler geometry. In 1990s, Donaldson \cite{Donaldson97} set up a program  to study the geometry of the space of K\"ahler metrics and its various applications to the well-known problems in K\"ahler geometry, notably the existence and uniqueness of Calabi's extremal K\"ahler metrics \cite{C1} (constant scalar curvature metrics). Donaldson's program and related problems have great impacts on the K\"ahler geometry. A key ingredient is the geodesic equation, which can be written as a homogeneous complex Monge-Ampere equation by the work of Semmes \cite{Semmes} and Donaldson \cite{Donaldson97}.
X.X. Chen \cite{C2000} proved the existence of $C^{1, \bar 1}$ geodesic given two boundary datum and solved a conjecture of Donaldson. Chen's solution has played an important role in Donaldson's program.

Inspired by the theory of the space of K\"ahler metrics,  Gursky-Streets introduced the fully nonlinear degenerate elliptic equation \eqref{GS1}, arising as the \emph{geodesic} equation of the Gursky-Streets metric. Gursky and Streets proved {a remarkable geometric consequence: solutions of the  $\sigma_2$-Yamabe problem--whose existence follows from their positivity assumption and Chang, Gursky and Yang \cite{CGY02}-are unique, unless the manifold is conformally equivalent to the round sphere. This is a surprising departure from the classical Yamabe problem, where explicit examples of nonuniqueness are known (see e.g. \cite{P93,Br1,Br2,Sch89,V2000} ). Therefore, the geodesic equation has played an essential and important role in the uniqueness of the solutions of the  $\sigma_2$-Yamabe problem.

 In \cite{He18} the first author solved the following equation, for a smooth $f>0$ and any $n\geq 4$,
\begin{equation*}\label{GS001}
u_{tt}\sigma_2(A_u)-\langle T_{1}(A_u),  \nabla u_t\otimes \nabla u_t\rangle=f
\end{equation*}
with the uniform $C^{1, 1}$ estimates,  and then  gave a strong solution to the degenerate equation
\[
u_{tt}\sigma_2(A_u)-\langle T_{1}(A_u),  \nabla u_t\otimes \nabla u_t\rangle=0
\]
by letting $f=s\rightarrow 0$. Using the uniform $C^{1, 1}$ estimates, the first author was able to carry out the approach introduced in \cite{GS-2017} to give a \emph{simplified} and more \emph{straightforward} proof of the uniqueness of  the $\sigma_2$-Yamabe problem \cite{He18}.

 Recall that for $\Lambda=(\lambda_1,...,\lambda_n) \in \R^n$, the $k$-th elementary symmetric function is defined as
$$ \sigma_k(\Lambda)=\sum_{i_1<...<i_k} \lambda_{i_1}\cdots\lambda_{i_k}.$$
The operator $\sigma_k$ has many useful and important properties, and itself is very significant. The $\sigma_k$-Hessian and curvature equations are developed by Caffarealli-Nirenberg-Spruck in \cite{CNS3}. It appears in the classical $\sigma_k$-Yamabe problem in conformal geometry, see e.g. \cite{LL03}. It also appears in the celebrated Christoffel-Minkowski problem in convex geometry, see e.g. \cite{GM1}.

In this paper we consider all the cases $n\geq 2k$, which generalizes substantially \cite{He18} for $k=2$.
At first, we establish the \emph{a priori estimates} of the following equations  for $u: M\times [0, 1]\rightarrow \cC^+_k$
\begin{equation}\label{GS2}
u_{tt}\sigma_k(A_u)-\langle T_{k-1}(A_u),  \nabla u_t\otimes \nabla u_t\rangle=f,
\end{equation}
where $f>0$  is  a given smooth function.  Specifically, we prove the following main result.
\vskip 10pt
\begin{thm}\label{thmm}Let $2k\leq n$. Given $u_0, u_1\in \cC^+_k$ and a smooth function $f>0$, there exists a unique smooth solution $u$ of \eqref{GS2}
such that  $u(0, \cdot)=u_0$, $u(1, \cdot)=u_1$. Moreover, we have the following uniform $C^{1, 1}$ a priori estimates,
\begin{equation}
|u|_{C^0}+|u_t|\leq C=C(C_2, \sup f),\; \max \left\{|\nabla u|, u_{tt}, |\nabla^2 u|, |\nabla u_t|\right\}\leq C_3.
\end{equation}
\end{thm}

\vskip 5pt

Taking $f=s$ for a positive constant $s$ and let  $s\rightarrow 0^{+}$,  we obtain the existence of solutions to degenerate equations. For the uniqueness results, we use maximum principle by constructing approximate solutions. At last, we conclude by a overall result as following.
\begin{thm}\label{thmm2}Let $2k\leq n$. Given $u_0, u_1\in \cC^+_k$, there exists a unique function $u\in C^{1, 1}\cap \overline{\cC^+_k}$ with $u(0, \cdot)=u_0$, $u(1, \cdot)=u_1$ such that
\[
u_{tt}\sigma_k(A_u)-\langle T_{k-1}(A_u),  \nabla u_t\otimes \nabla u_t\rangle=0
\]
in $L^\infty$ sense.
Moreover, we have the following uniform $C^{1, 1}$ control of $u$,
\begin{equation*}
|u|_{C^0}+|u_t|+|\nabla u|+ u_{tt}+|\nabla^2 u| +|\nabla u_t|\leq C_2.
\end{equation*}
\end{thm}

\vskip 5pt
Note that in the above theorems the following conventions of dependence of the constants are used. We use $C_1$ to denote a uniformly bounded positive constant depending only on $(M^n, g)$; $C_2$ to denote a uniformly bounded constant depending in addition on the boundary values $u_0, u_1$; and $C_3$ to denote a uniformly bounded constant depending in addition on $f$. An important feature  is that $C_3$ does not depend on $\inf f$, but rather on
\[
\{\sup f +f^{-1}\left(|\nabla f|+|f_t|+|f_{tt}|+|\Delta f|\right)\}.
\]
We use the notation $C=C(a_1, a_2, \cdots )$ to denote a uniform constant which depends on parameters $a_1, a_2, \cdots$.
The precise dependence of constants on the boundary value $u_0, u_1$ and $f$ can be easily traced in the proofs.

\vskip 5pt

It is worth mentioning that there are several different key points compared with  \cite{He18} when $k=2$ is assumed. To solve the Gursky-Streets equations with uniform $C^{1, 1}$ estimates for $2k\leq n$, our technique is to first show the \emph{concavity} of the operator. The approach in \cite{He18} for $k=2$ does not seem to work for the general case. The proof in this article  relies heavily on the Garding theory of hyperbolic polynomials and the results in the theory of real roots for (interlacing) polynomials.
This concavity certainly plays a substantial role in the \emph{a priori estimates}.

Comparing to the case  $k=2$, the general case is more involved technically and the computations are certainly more complicated.
We should emphasize that a main contribution is the interior $C^2$ estimates. In the case $n=2k$, this interior $C^2$ estimate is extremely delicate and technically very involved. Moreover, we establish the uniqueness of solution to the degenerate equation which are not obvious at all, even for $k=2$.
To achieve the uniqueness of the degenerate equation \[
u_{tt}\sigma_k(A_u)-\langle T_{k-1}(A_u),  \nabla u_t\otimes \nabla u_t\rangle=0,
\]
we use the notion of \emph{viscosity solution} together with our $C^{1, 1}$ estimates. The uniqueness does not follow directly from  general uniqueness results of viscosity solution in literature, at least to our knowledge. In additional, the nonlinear structure of $A_u$ and $2k\leq n$ play important roles in our arguments and indeed the proof is rather technical.

In the Appendix, we include briefly the geometric structure in $\cC_k^{+}$ and the Gursky-Streets equation as a geodesic equation when $M^{2k}$ is conformally flat. Moreover the functional introduced by Brendle and Viaclovsky \cite{BV04} is geodesically convex. As an application, we prove that if $M^{2k}$ is locally conformally flat and $\cC_k^+$ is not empty, then any solution of the $\sigma_k$ Yamabe problem is conformally diffeomorphic to the round sphere $S^{2k}$. This gives a totally different proof of a classical result of A.B. Li and Y. Y. Li \cite{LL03, LL05}.


\numberwithin{equation}{section}
\numberwithin{thm}{section}
\vskip25pt

\section{Concavity}
\vskip5pt
In this section we establish the convexity of the Gursky-Streets equations.
Denote the symmetric matrix $R=\left(
r_{ij}
\right)$
for $0\leq i, j\leq n$. We write $r=\left(r_{ij}\right)$ for $1\leq i, j\leq n$ and $x=\left(r_{01}, \cdots, r_{0n}\right)$.  For $r\in \Gamma_k^{+}$, define the operator
\[
F_{k}(R)=r_{00}\sigma_k(r)-\langle T_{k-1}(r), x\otimes x\rangle
\]
Denote the set $\cS$ to be the set of symmetric $(n+1)\times (n+1)$ matrices satisfying the following,
\[
\cS=\{R: r\in \Gamma_k^{+}, F_k(R)>0\}\]

\begin{thm}\label{thm2.1}
The set $\cS$ is a convex cone and  $F_{k}^{\frac{1}{k+1}}(R)$ (and hence $\log F_k(R))$ is concave on $\cS$.
\end{thm}

This proves the concavity of Gursky-Streets equations and it confirms a conjecture of the first named author \cite[Conjecture 4.3]{He18}.
When $k=n$, the operator $F_n=\det$ and the concavity is a well-known result. When $k=1$ the operator $F_1$ is the Donaldson operator and the concavity was proved by S. Donaldson and Chen-He.
The concavity of $\log F_k$ is equivalent to the convexity of $H_k(r, x)$ with
\[
H_k(r, x)=\frac{x T_{k-1}(r)x^T}{\sigma_k(r)}=x\left(\frac{\p \log \sigma_{k}(r)}{\p r_{ij}}\right)x^T
\]
It is well-known that $\log \sigma_k(r)$ is concave (over $\Gamma_k^+$). The convexity of $H_k$ (with $x$ fixed), implies that $\left(\frac{\p \log \sigma_{k}(r)}{\p r_{ij}}\right)$ as a function on $r$ is convex, which is about the positivity of ``third derivatives " of $\log\sigma_k$. In \cite{He18}, the first named author proved the convexity of $H_2$ by very involved computations. The method is elementary but consists of delicate and complicated computations. For $3\leq k\leq n-1$, the approach used in \cite{He18} seems to be too involved to prove the convexity of $H_k$.

In this paper we adopt an approach relying on Garding's theory of hyperbolic polynomials \cite{garding} to prove the concavity of  $F_k^{1/{k+1}}$ (and hence $\log F_k$). Using the theory of hyperbolic polynomials (see \cite[Proposition 2.1.31]{LH}), this is equivalent to proving the following

\begin{thm}\label{real01} For each $1\le k\le n $, the polynomial   $$p_{k+1}(t)=F_k(R+tI)=(r_{00}+t)\sigma_k(r+tI)-xT_{k-1}(r+tI)x^T$$ has only real roots for any real $n$-dimention vector $x$ and $n \times n$ real symmetric matrix $r$. Furthermore, if $\alpha_1\le \alpha_2\le \cdots\le  \alpha_{k+1}$ are the all real roots of $p_{k+1}$ , then they are separated by the $k$ real roots of $\sigma_k(r+tI)$.  We also have $\alpha_i\in [\min\limits_{1\le j\le n}\{-r_j\}, \max\limits_{1\le j\le n}\{-r_j\}]$, for $2\le i\le k-1$, where $(r_1, \cdots, r_n)$ are real eigenvalues of $r$.
\end{thm}

We need some facts about the convex cone $\Gamma^+_k$ and the Newton transformation $T_k(A)$. With the standard Euclidean metric, the $k$-th Newton transformation associated with a symmetric matrix $S$ (on $\R^n$) is given by
\[
T_k(S)=\sigma_k(S)I-\sigma_{k-1}(S) S+\cdots+(-1)^kS^k.
\]

\begin{prop}\label{gamma1}We have,
\begin{enumerate}
\item Each $\Gamma^+_k$ is an open convex cone.
\item If $A\in \Gamma^+_k$, then $T_{k-1}(A)$ is positive definite.
\item $\log \sigma_k$ and $\sigma_k^{1/k}$ are concave on $\Gamma^+_k$.
\end{enumerate}
\end{prop}

We also need the following fact (see \cite{GS-2017}),
\begin{prop}\label{gamma2}Given $A$ a symmetric matrix and $X$ a vector, then
\begin{equation}
\begin{split}
&\langle T_k(A-X\otimes X), X\otimes X\rangle =\langle T_k(A), X\otimes X\rangle\\
&\sigma_k(A-X\otimes X)=\sigma_k(A)-\langle T_{k-1}(A), X\otimes X\rangle
\end{split}
\end{equation}
\end{prop}

It is generally a hard problem to check whether a polynomial has all real roots. For example, when $k=2$, it is not easy to check that the cubic polynomial $F_2(R+tI)$ has all real roots by direct computations, even though the conditions on cubic polynomial having all real roots are well-known. Instead we exhibit the structure of real roots of $F_k(R+tI)$ and its relation with the polynomial $\sigma_k(r+tI)$. This allows us to use the theory of interlacing to assert that $F_k(R+tI)$ has all real roots.
Since $F_k(R)$ involves $r_{00}$, $x$ and only eigenvalues of $r$, we can diagonalize $r$ such that both of $r+tI$ and  $T_{k-1}(r+tI)$ are diagonal matrices. We denote the eigenvalues of $r$ by $(r_1, \dots, r_n)$ and write
$$\left\{
  \begin{array}{ll}
    \sigma_k(r+tI)=\sum\limits_{1\le i_1<i_2<\cdots<i_k\le n}(t+r_{i_1})(t+r_{i_2})\cdots(t+r_{i_k})  \\
      \\
    xT_{k-1}(r+tI)x^T=\sum\limits_{1\le i\le n}q_{i,k}(t)x_i^2
  \end{array}
\right.$$
where $q_{i,k}(t)\triangleq \sum\limits_{0\le j\le k-1}(-1)^j\sigma_{k-1-j}(r+tI)(t+r_i)^j$ and $\sigma_0(r+tI)$ is defined to be 1. Thus we have $$p_{k+1}(t)=(r_{00}+t)\sigma_k(r+tI)-\sum\limits_{1\le i\le n}q_{i,k}(t)x_i^2.$$


We introduce some notations for the convenience in the following  discussion.

\begin{enumerate}
 \item [(1)] {\bf{RZ}}$\triangleq$ the set of univariate polynomials with all real zeros.
  \item [(2)] $\pi_n(t)\triangleq (t+r_1)(t+r_2)\cdots(t+r_n)$;
  \item [(3)] $\pi_{i,n}(t)\triangleq \frac{\pi_n(t)}{t+r_i}$.
\end{enumerate}

Firstly we give two formulas about $\sigma_k(r+tI)$ and $q_{i,k}$, which will be used later.

\begin{prop}\label{pp}
 Let $\pi_n^{(i)}$ denote the $i$th derivative of $\pi_n$ with respect to $t$. Then
 \[
 \sigma_k(r+tI)=\dfrac{\pi_n^{(n-k)}}{(n-k)!}.
 \]
\end{prop}
\begin{proof}This is straightforward computation, \begin{eqnarray}
\dfrac{\pi_n^{(n-k)}}{(n-k)!}&=& \frac{1}{(n-k)!}\cdot\sum\limits_{i_1+\cdots+i_n=n-k}\frac{(n-k)!}{i_1!\cdots i_n!}(t+r_1)^{(i_1)}\cdots (t+r_n)^{(i_n)}\nonumber \\
&=&\sum\limits_{1\le i_1< \cdots<i_k\le n}(t+r_{i_1})(t+r_{i_2})\cdots (t+r_{i_k})\nonumber \\
&=&\sigma_k(r+tI)\nonumber
\end{eqnarray}
\end{proof}

\begin{prop}\label{pp2}
$q_{i,k}=\dfrac{\pi^{(n-k)}_{i,n}}{(n-k)!}$. The superscript $(n-k)$ still stands for the $(n-k)$-th derivative.
\end{prop}

\begin{proof}
We have the following,
\begin{eqnarray}
 q_{i,k}&=&\sum\limits_{0\le j\le k-1}(-1)^j\sigma_{k-1-j}(r+tI)(t+r_i)^j \nonumber\\
q^{(1)}_{i,k} &=&\sum\limits_{0\le j\le k-2}(-1)^j\frac{\pi_n^{(n-k+2+j)}}{(n-k+1+j)!}(t+r_i)^j\nonumber
+\sum\limits_{1\le j\le k-1}(-1)^j\frac{\pi_n^{(n-k+1+j)}}{(n-k+1+j)!}j(t+r_i)^{j-1}\nonumber \\
&=&\sum\limits_{0\le j\le k-2}(-1)^j(t+r_i)^{j}\pi_n^{(n-k+2+j)}\bigg(\frac{1}{(n-k+1+j)!}-\frac{j+1}{(n-k+2+j)!}\bigg)\nonumber\\
&=&\sum\limits_{0\le j\le k-2}(-1)^j(t+r_i)^{j}\frac{\pi_n^{(n-k+2+j)}}{(n-k+2+j)!}\big(n-k+2+j-(j+1)\big)\nonumber\\
&=&(n-k+1)\sum\limits_{0\le j\le k-2}(-1)^j(t+r_i)^{j}\sigma_{k-2-j}(r+tI)\nonumber\\
&=&(n-k+1)q_{i,k-1}\nonumber
\end{eqnarray}
while \begin{eqnarray}q_{i,n} &=&\sum\limits_{0\le j\le n-1}(-1)^j\sigma_{n-1-j}(r+tI)(t+r_i)^j\nonumber\\
&=&\sum\limits_{0\le j\le n-1}(-1)^j(t+r_i)^j\dfrac{\big(\pi_{i,n}\cdot (t+r_i)\big)^{(1+j)}}{(j+1)!}\nonumber \\
&=&\sum\limits_{0\le j\le n-1}(-1)^j\dfrac{(t+r_i)^j}{(j+1)!}\big(\pi_{i,n}^{(1+j)}(t+r_i)+(j+1)\pi_{i,n}^{(j)}\big)\nonumber \\
&=&\sum\limits_{1\le j\le n }\dfrac{-(-1)^j}{j!}(t+r_i)^j\pi_{i,n}^{(j)}+\sum\limits_{0\le j\le n-1}\dfrac{(-1)^j}{j!}(t+r_i)^j\pi_{i,n}^{(j)}\nonumber \\
&=&\pi_{i,n}-\dfrac{(-1)^n}{n!}(t+r_i)^n\pi_{i,n}^{(n)}\nonumber\\
&=&\pi_{i,n} .  \nonumber
\end{eqnarray}
The last equality holds because $\pi_{i,n}$ is a polynomial of degree $n-1$.  So we have
$$q_{i, k}=\dfrac{q^{(1)}_{i,k+1} }{n-k}=\dfrac{q^{(2)}_{i,k+2} }{(n-k)(n-k-1)}=\cdots=\dfrac{q^{(n-k)}_{i,n} }{(n-k)!}=\dfrac{\pi^{(n-k)}_{i,n} }{(n-k)!}.$$
\end{proof}

By the above-mentioned formulas for  $\sigma_k(r+tI)$ and $q_{i,k}$ , we have:
\begin{lemma}\label{cct1}
$\sigma_k(r+tI)$  and $q_{i,k}(t)$ are real rooted for each $1\le i, k \le n$.
\end{lemma}
\begin{proof} Obviously, $\pi_n(t)=(t+r_1)(t+r_2)\cdots(t+r_n)$ has $n$ real roots $\{-r_1,-r_2,\cdots,-r_n\}$ by counting multiplicity. By Rolle's Mediate Value Theorem,  the $i$th derivative of $f$  has exactly $n-i$ real roots, for any $0\le i\le n$. Thus $\sigma_k(r+tI)$  has $n-(n-k)=k$ real roots by Proposition \ref{pp}. While it is  a polynomial of $t$ of degree  $k$. Therefore  $\sigma_k(r+tI)$ is real rooted.

Similarly, $\pi_{i,n}$ has  $n-1$ real roots $\{-r_1,-r_2,\cdots, -\hat{r_i}, -\cdots, -r_n\}$. Then $q_{i,k}$  has exactly $n-1-(n-k)=k-1$ real roots by Proposition \ref{pp2} and it is of degree $k-1$. Then all the roots are real.
\end{proof}

Assume $f,g$ are two polynomials in RZ and let $\{u_i\},\{v_j\}$ be all roots of $f, g$ in nonincreasing order respectively.  We say that $g$ interlaces $f$ , denoted by $g \preccurlyeq f$, if either deg$f=$deg $g = n$ and$$v_n\le u_n\le  v_{n-1} \le \cdots\le  v_2\le  u_2\le  v_1\le  u_1 ;$$  or deg $f  = $ deg $g + 1 = n$ and$$u_n\le  v_{n-1} \le \cdots\le  v_2\le  u_2\le  v_1\le  u_1 .$$
If all inequalities  are strict, then we say that $g$ strictly interlaces $f$, denoted by $g\prec f $. Interlacing describes the relative positions of the real roots of a pair of polynomials. It has shown its power in studying real rootedness of polynomials.
We have the following important relation between $\sigma_k(r+tI)$ and $q_{i,k}$ :

\begin{lemma}\label{cct2}
$q_{i,k}$ interlaces $\sigma_k(r+tI)$ for any $1\le i,k\le n$.
\end{lemma}
\begin{proof} By Propositions \ref{pp} and  \ref{pp2}, it's equivalent to prove $\pi^{(n-k)}_{i,n}(t)$ interlaces $\pi^{(n-k)}_{n}(t)$.  This holds definitely because for each $1\le i\le n$,  $\pi_{i,n}(t)$ interlaces  $\pi_{n}(t)$ and the differential operator $\frac{d}{dt}$ preserves interlacing  \cite[Theorem 1.47, page 38/779]{st}.
\end{proof}

Now we come to prove the real rootedness of $p_{k+1}$.  The main idea of the proof depends on the following theorem \cite[Theorem 2.3, page 5/19]{lili},  which is very useful to prove the real rootedness of some specific categories of real polynomials.

\begin{thm}\cite{lili}:  Let $F, f, g_1, \cdots, g_n$ be real polynomials satisfying the following conditions:\begin{enumerate}
\item $F(t)=a(t)f (t) + b_1 (t) g_1(t) + \cdots + b_n (t) g_n (t)$ , where $a, b_i$ are real  polynomials such that $\text{deg}\,F=degf$ or $\text{deg}f+1$.
\item  $f, g_i\in RZ$ and $g_i$ interlaces $f$ for each $i$.
\item $F$ and $g_i$ have leading coefficients of the same sign.
\item $b_i(z)\le 0$ for each $i$ and each zero $z$ of $f$ .
\end{enumerate}
Then $F\in RZ$ and $f$ interlaces $F$. In particular, if for each zero $z$ of $f$ , there is an index $i$ such
that $g_i$ strictly interlace $f$ and $b_i(z) < 0$, then $f$ strictly interlaces $F$.
\end{thm}

By setting
$$F(t)=p_{k+1}(t),\ \ a(t)=t+r_0,\ \  f(t)=\sigma_k(r+tI), \ \ g_i(t)=q_{i,k}(t), \ \  b_i=-x_i^2.$$  in the above theorem, we can see :
\begin{description}
  \item[(1)] deg($p_{k+1}$)$=k+1=$ deg$\bigg(\sigma_k(r+tI)\bigg)$+1
  \item[(2)] $\sigma_k(r+tI), q_{i,k}\in RZ$  and $q_{i,k}$ interlaces $\sigma_k(r+tI)$ for  each  $i$ and $k$ (Lemma \ref{cct1} and \ref{cct2}).
  \item[(3)] $p_{k+1}(t)$ and $q_{i,k}(t)$ have both positive leading coefficients. In fact, the leading coefficient of $p_{k+1}$ equals $\frac{n!}{(n-k)!^2}$ and that of $q_{i,k}$ equals $\frac{(n-1)!}{(n-k)!^2}$.
  \item[(4)]$-x_i^2\le 0$ for each $i$ and each zero of $\sigma_k(r+t I)$
\end{description}
\noindent Thus, all the four required conditions are all satisfied and then we get the real rootedness of $p_{k+1}(t)=F_k(R+tI)$.
Hence by {\cite[Proposition 2.1.31]{LH}}, $F_k^{1/(k+1)}$ is concave on the component of $I$ in the set of $\{R: F_k(R)\neq 0\}$, which is a connected convex cone. We denote this convex cone as $\tilde \cS$. Since $I$ is in  $\cS$ and $\cS$ is connected such that $F_k(R)>0$ on $\cS$,  hence $\cS$ is contained in $\tilde \cS$. The concavity of $F_k^{1/(k+1)}$ on $\tilde \cS$ (and hence on $\cS$) implies that $\cS$ is indeed a convex cone,
which proves Theorem \ref{real01}. Even though we do not really need this fact, it is a standard practice to show that  $\cS=\tilde \cS$.

\section{Existence and uniqueness}
In this section we establish the a priori estimates to solve the equation.
First we recall Gursky-Streets' equations and related notations briefly.
Let $(M^n, g)$ be a compact Riemannian manifold with the conformal class $[g]$, $n\geq 3$.
The metrics in $[g]$ can be parametrized by metrics of the form $g_u=e^{-2u}g$. The Ricci curvature is given by
\[
Ric(g_u)=Ric+(n-2)\left(\nabla^2 u+\nabla u\otimes \nabla u-\frac{1}{2}|\nabla u|^2g\right)+\left(\Delta u-\frac{n-2}{2}|\nabla u|^2\right) g
\]
and the scalar curvature is given by
\[
R(g_u)=e^{2u}\left(R+2(n-1)\left(\Delta u-\frac{n-2}{2}|\nabla u|^2\right)\right)
\]
Under the conformal change, the Schouten tensor is given by
\[
A_u=A(g_u)=A+\nabla^2 u+\nabla u\otimes \nabla u-\frac{1}{2}|\nabla u|^2 g.
\]

In  this section we derive the \emph{a priori estimates} to solve the equation.
Denote the operator
\begin{equation}\label{eqf1}
F_k(u_{tt}, A_u, \nabla u_t)=u_{tt}\sigma_k(A_u)-\langle T_{k-1}(A_u),  \nabla u_t\otimes \nabla u_t\rangle.
\end{equation}

Given $u_0, u_1\in C^\infty$ such that $A_{u_0}, A_{u_1}\in \Gamma^+_k$,  we assume that $u\in C^\infty$ such that $A_u\in \Gamma^+_k$, and solves the equation
\begin{equation}\label{eqf2}
F_k(u_{tt}, A_u, \nabla u_t)=f,
\end{equation}
for a positive function $f\in C^\infty$, with the boundary condition $u(0, x)=u_0(x)$, $u_1(1, x)=u_1(x).$
\vspace{3mm}

We need the following standard formulas of elementary symmetric functions, see e.g. \cite{Lieb}.
\begin{prop}\label{prop3.1}Let $\lambda =(\lambda_1,...,\lambda_n)\in \mathbb{R}^n$ and $k=0,1,...,n$, then
 \begin{align}\label{}
 &\frac{\partial \sigma_{k+1}(\lambda)}{\partial \lambda_i}=\sigma_{k}(\lambda|i), \quad \forall \, 1 \le i \le n,\\
 &\sigma_{k+1}(\lambda)=\sigma_{k+1}(\lambda|i)+\lambda_i\sigma_k(\lambda|i), \quad \forall \, 1 \le i \le n,\\
 &\sum_{i=1}^{n}\lambda_i \sigma_k(\lambda|i)=(k+1)\sigma_{k+1}(\lambda),\\
 &\sum_{i=1}^{n} \sigma_k(\lambda|i)=(n-k)\sigma_{k}(\lambda).
 \end{align}
 Here we denote by $\sigma_k(\lambda|i)$
  the symmetric function with $\lambda_i=0$.
\end{prop}

\begin{prop}\label{NI}(Newton's inequality). For any $k \in \{1,2,...,n-1\},$ and $\lambda \in \mathbb{R}^n$, we have
 \begin{align}\label{eqN1}
 (n-k+1)(k+1)\sigma_{k-1}(\lambda)\sigma_{k+1}(\lambda) \le k(n-k)\sigma_k^2(\lambda).
\end{align}
In particular, we have
 \begin{align}\label{eqN2}
 \sigma_{k-1}(\lambda)\sigma_{k+1}(\lambda) \le \sigma_k^2(\lambda).
\end{align}
\end{prop}

\begin{prop}\label{MI}(MacLaurin inequality). For  $\lambda \in \Gamma^+_k$ and $ k \ge l \ge 1$, we have
 \begin{align}\label{eqM1}
 \left[ \frac{\sigma_{k}(\lambda)}{C_n^k} \right ]^{\frac{1}{k}}\le \left[\frac{\sigma_{l}(\lambda)}{C_n^l} \right ]^{\frac{1}{l}}.
\end{align}
\end{prop}

\begin{prop}\label{NMI}(Generalized Newton-MacLaurin inequality).
For $\lambda \in \Gamma_k^+$ and $k > l \geq 0$, $ r > s \geq 0$, $k \geq r$, $l \geq s$, we have
\begin{align} \label{1.2.6}
\Bigg[\frac{{\sigma _k (\lambda )}/{C_n^k }}{{\sigma _l (\lambda )}/{C_n^l }}\Bigg]^{\frac{1}{k-l}}
\le \Bigg[\frac{{\sigma _r (\lambda )}/{C_n^r }}{{\sigma _s (\lambda )}/{C_n^s }}\Bigg]^{\frac{1}{r-s}},
\end{align}
and the equality holds if and only if $\lambda_1 = \lambda_2 = \cdots =\lambda_n >0$.
\end{prop}

We set
\begin{equation}
 E_u=u_{tt}A_u-\nabla u_t\otimes \nabla u_t.
\end{equation}
Then by Proposition~\ref{gamma2}, the operator in equation \eqref{eqf1} is transformed into
\begin{equation}
F_k(u_{tt}, A_u, \nabla u_t)= u_{tt}^{1-k}\sigma_k(E_u).
\end{equation}

\begin{prop}
\label{Lem1}
 Suppose $A_u \in \Gamma_k^+$ and $u_{tt} >0$, $\sigma_k(E_u)>0$, then $E_u \in \Gamma^+_k$. In particular, we have
 \begin{align}\label{}
 \sigma_{i-1}(E_u) \ge f_i \sigma_i(A_u)^{-1}\sigma_{i-1}(A_u)u_{tt}^{i-2},  \quad 2 \le i \le \ k
\end{align}
with $f_i = u_{tt}^{1-i}\sigma_i(E_u)$, $f_k=f$.
\end{prop}
\begin{proof}
Since $\sigma_k(E_u)>0$, it is  sufficient to prove if $\sigma_i(E_u) > 0 $,  then $\sigma_{i-1}(E_u) > 0 $ for all $2 \le i \le \ k.$

Note
\begin{align}\label{fi}
 f_i=  u_{tt}^{1-i}\sigma_i(E_u)= u_{tt}\sigma_i(A_u)-\langle T_{i-1}(A_u),\nabla u_t\otimes \nabla u_t\rangle.
 \end{align}
 It yields
 \begin{align}\label{}
  u_{tt}=f_i\sigma_i(A_u)^{-1}+\langle T_{i-1}(A_u),\nabla u_t \otimes \nabla u_t\rangle\sigma_i(A_u)^{-1}.
 \end{align}
We compute
\begin{align}
 \begin{split}
  &\sigma_{i-1}(E_u)\\
  =&u_{tt}^{i-1}\sigma_{i-1}(A_u)-u_{tt}^{i-2}\langle T_{i-2}(A_u),\nabla u_t \otimes \nabla u_t\rangle\\
  =&[f_i\sigma_i(A_u)^{-1}+\langle  T_{i-1}(A_u),\nabla u_t \otimes \nabla u_t\rangle\sigma_i(A_u)^{-1}]u_{tt}^{i-2}\sigma_{i-1}(A_u)\\
  &-u_{tt}^{i-2}\langle T_{i-2}(A_u),\nabla u_t \otimes \nabla u_t\rangle\\
  =&f_i \sigma_i(A_u)^{-1}\sigma_{i-1}(A_u)u_{tt}^{i-2}+ u_{tt}^{i-2}\langle\dfrac{\sigma_{i-1}(A_u)}{\sigma_i(A_u)}T_{i-1}(A_u)-T_{i-2}(A_u),\nabla u_t \otimes \nabla u_t\rangle
\end{split}
 \end{align}
 So we only need to show
 \begin{align}\label{}
   \dfrac{\sigma_{i-1}(A_u)}{\sigma_i(A_u)}T_{i-1}(A_u)-T_{i-2}(A_u)
 \end{align}
 is positive definite.

 Diagonalize $A_u$ with eigenvalues $\tilde{\lambda}=(\tilde{\lambda_1},...,\tilde{\lambda_n})$, and denote by $\sigma_k(\tilde{\lambda}|l)$
  the symmetric function with $\tilde{\lambda}_l=0$. Then we need to verify $\forall \;l \;\in \{1,...n\}$
  \begin{align}\label{}
   \dfrac{\sigma_{i-1}(\tilde{\lambda})}{\sigma_i(\tilde{\lambda})}\sigma_{i-1}(\tilde{\lambda}|l)-\sigma_{i-2}(\tilde{\lambda}|l) \ge 0.
 \end{align}
It is true by
\begin{align}
 \begin{split}
   &\sigma_{i-1}(\tilde{\lambda})\sigma_{i-1}(\tilde{\lambda}|l)-\sigma_i(\tilde{\lambda})\sigma_{i-2}(\tilde{\lambda}|l)\\
   =& [\sigma_{i-1}(\tilde{\lambda}|l)+\tilde{\lambda}_l\sigma_{i-2}(\tilde{\lambda}|l)]\sigma_{i-1}(\tilde{\lambda}|l)
   -[\sigma_{i}(\tilde{\lambda}|l)+\tilde{\lambda}_l\sigma_{i-1}(\tilde{\lambda}|l)]\sigma_{i-2}(\tilde{\lambda}|l)\\
   =&\sigma_{i-1}^2(\tilde{\lambda}|l) -\sigma_{i}(\tilde{\lambda}|l)\sigma_{i-2}(\tilde{\lambda}|l)\\
   \ge& 0.
\end{split}
 \end{align}
The last inequality holds from Newton's inequality \eqref{eqN2}.

\end{proof}

In the following for simplicity we denote the symmetric tensor product as follows,
\[
X\boxtimes Y:=X\otimes Y+Y\otimes X.
\]

\begin{prop}\label{proplv}
Given $u\in C^2$ such that $A_u\in \Gamma^+_k$, then the equation \eqref{eqf2} is strictly elliptic when $f>0$.
The linearized operator is given by
\begin{equation}\label{E3}
 \begin{split}
\mathcal{L}_{F_k}(v)= &(1-k)u_{tt}^{-k}\sigma_k(E_u)v_{tt} + u_{tt}^{1-k}\langle T_{k-1}(E_u), v_{tt}A_u + u_{tt}\mathcal{L}_{A_u}(v)- \nabla u_t\boxtimes \nabla v_t \rangle\\
=& u_{tt}^{-1}fv_{tt} +  u_{tt}^{1-k}\langle T_{k-1}(E_u),  u_{tt}\mathcal{L}_{A_u}(v)- \nabla u_t\boxtimes \nabla v_t +u_{tt}^{-1}v_{tt} \nabla u_t\otimes \nabla u_t \rangle .
\end{split}
 \end{equation}
where $\cL_{A_u}(v)$ is the linearization of $A_u$, given by
$$\mathcal{L}_{A_u}(v)=\nabla^2 v+ \nabla  u \boxtimes\nabla v - \langle \nabla u , \nabla v \rangle g.$$
\end{prop}

\begin{proof}
First note that when $f>0$, by the assumption $A_u\in \Gamma^+_k$, we have $u_{tt}>0$.
Suppose $\delta u=v$, and we use the variation of $\sigma_k$,
$\delta \sigma_k(E_u)=\langle T_{k-1}(E_u), \delta E_u\rangle$. By direct computation we have
\[
\cL_{F_k}(v)=u_{tt}^{1-k}\langle T_{k-1}(E_u), v_{tt} A_u+u_{tt}\cL_{A_u}(v)-\nabla u_t\otimes \nabla v_t-\nabla v_t\otimes \nabla u_t\rangle -(k-1)u_{tt}^{-k}\sigma_k(E_u)v_{tt}.\]
To show the ellipticity, we only need to take care of second order derivatives of $v$. The leading terms reads,
\[
u_{tt}^{1-k}\langle T_{k-1}(E_u), v_{tt} A_u+u_{tt} \nabla^2 v-\nabla u_t\otimes \nabla v_t-\nabla v_t\otimes \nabla u_t\rangle -(k-1)u_{tt}^{-k}\sigma_k(E_u) v_{tt}
\]
Replacing the derivatives of $(v_t, \nabla v)$ by a vector $(\xi, X)\in T([0, 1]\times M)=\R\times \R^{n}$, we need to show that the following quadratic form is positive definite,
\[
Q(\xi, X):=\langle T_{k-1}(E_u), \xi^2 A_u+u_{tt} X\otimes X-\xi \nabla u_t\otimes X-\xi X\otimes \nabla u_t\rangle -(k-1)u_{tt}^{-1}\sigma_k(E_u) \xi^2
\]
We compute
\[
\begin{split}
&\xi^2 A_u+u_{tt} X\otimes X-\xi \nabla u_t\otimes X-\xi X\otimes \nabla u_t\\
=&\xi^2(A_u-u_{tt}^{-1}\nabla u_t\otimes \nabla u_t)+Y\otimes Y\\
=&u^{-1}_{tt} \xi^2 E_u+Y\otimes Y
\end{split}
\]
where $Y=\sqrt{u_{tt}}X-\xi \nabla u_t$. It follows that
\[
Q(\xi, X)=\langle T_{k-1}(E_u), Y\otimes Y\rangle+ u^{-1}_{tt}\left(\langle T_{k-1}(E_u), E_u\rangle-(k-1)\sigma_k(E_u)\right) \xi^2.
\]
Since $E_u\in \Gamma^{+}_k$ from Proposition~\ref{Lem1}, $T_{k-1}(E_u)$ is positive definite. A direct computation by Proposition \ref{prop3.1} gives
\begin{equation}\label{Eu1}
\langle T_{k-1}(E_u),  E_u\rangle-(k-1)\sigma_k(E_u)=\sigma_k(E_u)>0
\end{equation}
It then follows that, for $(\xi, X)\neq 0$,
$Q(\xi, X)>0$.
To show the second identity in \eqref{E3}, we compute
\[
\langle T_{k-1}(E_u), v_{tt} A_u\rangle= u_{tt}^{-1}\langle T_{k-1}(E_u), E_u\rangle v_{tt}+\langle T_{k-1}(E_u), u_{tt}^{-1} v_{tt} \nabla u_t\otimes \nabla u_t\rangle.
\]
Applying \eqref{Eu1} again we get the result.
This completes the proof.
\end{proof}

We will need the comparison function as follows. Denote $U_a=at(1-t)+(1-t)u_0+t u_1$ for any number $a$. Note that $U_0=u_0$ at $t=0$, $U_1=u_1$ at $t=1$ for any $t$. In particular $U_a$ has the same boundary value with $u$.
If $u_0, u_1$ are admissible ($A_{u_i}\in \Gamma^+_k$ for $i=0, 1$), $U_0=(1-t)u_0+t u_1$ is admissible \cite{Via} and hence $U_a$ are all admissible, since we have $A_{U_0}=A_{U_a}, (\nabla U_0)_t=(\nabla U_{a})_t$ for any $a$.
Gursky-Streets \cite{GS-2017} proved a uniform $C^1$ estimate for the equation
\[
u_{tt}^{1-k}\sigma_k(E_u^\epsilon)=f,
\]
where $E_u^\epsilon=(1+\epsilon)u_{tt}A_u-\nabla u_t\otimes \nabla u_t$, for any $k\geq 1$. They introduced an extra $\epsilon$-parameter for the purpose of $C^{1, 1}$ estimates, which does not play any essential role in $C^1$ estimates. Hence their results clearly apply in our setting to obtain uniform $C^1$ estimates and most computations required in $C^1$ estimates can be found in \cite{GS-2017}. Nevertheless we will include details of $C^1$ estimates for completeness
(these computations will be needed for uniform $C^{1, 1}$ estimates). Our arguments below are a variant of the case $k=2, n\geq 4$, explored in \cite{He18}.

\subsection{$C^0$ estimates}
In this section we derive the $C^0$ estimates.

\begin{prop}\label{c0}There exists $a=a(u_0, u_1, \sup f)>0$ sufficiently large, such that
\begin{equation*}
U_{-a}\leq u\leq U_0=(1-t)u_0+t u_1.
\end{equation*}
\end{prop}

\begin{proof}First by $u_{tt}>0$, we have
\[
\frac{u(\cdot, t)-u(\cdot, 0)}{t-0}<\frac{u(\cdot, 1)-u(\cdot, t)}{1-t}
\]
That gives the upper bound,
\[
u(\cdot, t)<(1-t)u(\cdot, 0)+tu(\cdot, 1)=(1-t)u_0+t u_1.
\]
We claim $u-U_{-a}\geq 0$ for $a>0$ sufficiently large. We argue by contradiction.
Since $u-U_{-a}=0$ for $t=0$ and $t=1$,  there exists an interior point $p=(t, x)\in (0, 1)\times M$, such that $u-U_{-a}$ obtains its minimum at $p$. Denote $u^s=su+(1-s)U_{-a}$ and $v=\p_s u^s(s=1)=u-U_{-a}$. Then $D^2v\geq 0$  and $\nabla v=0$ at $p$. By the concavity of $\log F_k$, it follows that for $s\in [0, 1]$,
\[
\log F_k(u^s_{tt}, A_{u^s}, \nabla u^s_t)\geq s\log F_k(u_{tt}, A_u, \nabla u_t)+(1-s)\log F_k({(U_{-a})}_{tt}, A_{U_{-a}}, \nabla {(U_{-a})}_{t})
\]
At $s=1$, we get (at $p$),
\begin{equation}\label{zero1}
{F_k}^{-1}\cL_{F_k} (v)\leq \log F_k(u_{tt}, A_u, \nabla u_t)-\log F_k({(U_{-a})}_{tt}, A_{U_{-a}}, \nabla {(U_{-a})}_{t}),
\end{equation}
where ${F_k}^{-1}\cL_{F_k}$ takes values at $u\; (s=1)$.
We can choose $a$ large enough such that $$F_k({(U_{-a})}_{tt}, A_{U_{-a}}, \nabla {(U_{-a})}_{t})=2a \sigma_k(A_{U_{0}})-(T_{k-1}(A_{U_{0}}), \nabla {(U_{0})}_{t}\otimes \nabla {(U_{0})}_{t})$$ is sufficiently large. Then the right hand side of \eqref{zero1} is negative (at $p$) since $ F_k(u_{tt}, A_u, \nabla u_t)=f$.  This is a contradiction given the claim that $\cL_{F_k}(v)\geq 0$ at $p$. Note that $D^2 v\geq 0, \nabla v=0$ at $p$. We compute, using $\nabla v=0$ at $p$,
\[
\cL_{F_k}(v)=u_{tt}^{1-k}\langle T_{k-1}(E_u), v_{tt} A_u+u_{tt} \nabla^2 v-\nabla u_t\otimes \nabla v_t-\nabla v_t\otimes \nabla u_t\rangle +(1-k)u_{tt}^{-k}\sigma_k(E_u) v_{tt}
\]
We can assume $v_{tt}>0$. Otherwise we have $v_{tt}=0$, then $u_{tt}=2a>0$. Notice that $\nabla v_t=0$ and $\nabla^2 v\geq 0$ (since $D^2 v\geq 0$ at $p$). In this case the claim follows trivially. If $v_{tt}>0$, the argument follows similarly as in Proposition \ref{proplv}. Indeed we write
\[
\cL_{F_k}(v)=u_{tt}^{1-k}\langle T_{k-1}(E_u), \frac{v_{tt}}{u_{tt}}E_u+Y\otimes Y+u_{tt}(\nabla^2v-v_{tt}^{-1}\nabla v_t\otimes \nabla v_t)\rangle +(1-k)u_{tt}^{-k}\sigma_k(E_u) v_{tt},
\]
where $Y=\sqrt{v_{tt}/u_{tt}} \nabla u_t-\sqrt{u_{tt}/v_{tt}}\nabla v_t$. By \eqref{Eu1}
and the positivity of $\nabla^2v-v_{tt}^{-1}\nabla v_t\otimes \nabla v_t$ (this is because $D^2 v\geq 0$), it follows that $\cL_{F_k}(v)\geq 0$.
\end{proof}

\subsection{$C^1$ estimates}First we have the following,

\begin{prop}\label{pro3.8}\label{c1t}Let $a$ be the constant in Proposition \ref{c0}. Then we have,
\[
-a+u_1-u_0\leq u_t\leq a+u_1-u_0
\]
\end{prop}
\begin{proof}
Since $u_{tt}>0$, it follows that $u_t(t, x)$ is increasing in $t$. Hence we only need to argue $u_t(0, x)\leq u_t(1, x)$ are both bounded. We compute, using Proposition \ref{c0},
\[
u_t(0, \cdot)=\lim_{t\rightarrow 0}\frac{u(t, \cdot)-u(0, \cdot)}{t}\geq \lim_{t\rightarrow 0} \frac{at(t-1)+t(u_1-u_0)}{t}=-a+u_1-u_0.
\]
It is evident that $u_t(0, \cdot)\leq u_1-u_0$ by convexity. Similarly we have $u_1-u_0\leq u_t(1, \cdot)\leq a+u_1-u_0$.

\end{proof}

To derive estimates of $|\nabla u|^2$ and second order derivatives, we need some preparation due to the complicated computations.
First we need to choose a normalization condition. Note that if $u$ is admissible, then $\tilde u=u-c_1 t-c_2$ is also admissible since $A_u, E_u$ do not change at all. Hence if $u$ is a solution, $\tilde u=u-c_1t-c_2$ is also a solution since  $\nabla \tilde u=\nabla u, D^2\tilde u=D^2 u$. The corresponding boundary condition is changed by a constant with $\tilde u_0=u_0-c_2, \tilde u_1=u_1-c_1-c_2$ and  $\tilde u_t=u_t-c_1$. Hence we can choose two sufficiently large constants $c_1$ and $c_2$ such that $\tilde u\leq -1$, and $\tilde u_t\leq -1$.  We can choose such a normalization condition on $u_0, u_1$ such that,
\begin{equation}\label{normalization}
-c_0\leq u\leq -1, -c_0\leq u_t\leq -1,
\end{equation}
where $c_0$ is the uniform bound we have obtained above for $|u|$ and $|u_t|$. \\

Next we compute $\cL_{F_k}(v)$ for various barrier functions $v$. The philosophy is well-known in nonlinear elliptic theory, to construct various barrier functions $v$ such that
\[
\cL_{F_k}(v)\geq -C+\text{good positive terms}
\]
Such barrier functions serve as the purpose of  \emph{subsolutions} with respect to $\cL_{F_k}$ and play an essential role in the maximum principle argument.
The first such function is the $t$-functions,
\begin{prop}
Suppose $v=v(t)$ is a $t$-function, then
\begin{equation}
\begin{split}
\cL_{F_k}(v)=&\frac{v_{tt}}{u^k_{tt}}\left(\langle T_{k-1}(E_u), u_{tt}A_u\rangle-(k-1)\sigma_k(E_u)\right)\\
=&\frac{v_{tt}}{u_{tt}^k}\left(\langle T_{k-1}(E_u), E_u+\nabla u_t\otimes \nabla u_t\rangle-(k-1)\sigma_k(E_u)\right)\\
=&\frac{v_{tt}}{u_{tt}^{k}}\left(\sigma_k(E_u)+\langle T_{k-1}(E_u), \nabla u_t\otimes \nabla u_t\rangle \right)\\
=&v_{tt}\sigma_k(A_u),
\end{split}
\end{equation}
where we apply Proposition \ref{gamma2} in the last step above.
In particular, \begin{equation}\label{t}
\cL_{F_k}(t^2)=2\sigma_k(A_u)\end{equation}
\end{prop}

The second choice is the function $-u$ itself.
We compute $\cL_{F_k}(u)$.
\begin{prop}\label{propu}We have,
\begin{equation}\label{u1}
\cL_{F_k}(u)=(k+1)u_{tt}^{1-k}\sigma_k(E_u)+u_{tt}^{2-k}\langle T_{k-1}(E_u), -A + \nabla  u \otimes\nabla u -\dfrac{1}{2}|\nabla u|^2 g \rangle.
\end{equation}
\end{prop}
\begin{proof}By \eqref{E3}, we compute
\[
 \begin{split}
&\mathcal{L}_{F_k} (u)\\
=&(1-k)u_{tt}^{1-k}\sigma_k(E_u) + u_{tt}^{1-k}\langle T_{k-1}(E_u), u_{tt}A_u + u_{tt}(\nabla^2 u + 2\nabla  u \otimes\nabla u -|\nabla u|^2 g)- 2\nabla u_t\otimes \nabla u_t \rangle\\
=&(1-k)u_{tt}^{1-k}\sigma_k(E_u) + u_{tt}^{1-k}\langle T_{k-1}(E_u),2E_u- u_{tt}A_u + u_{tt}(\nabla^2 u + 2\nabla  u \otimes\nabla u -|\nabla u|^2 g) \rangle\\
=& (1-k)u_{tt}^{1-k}\sigma_k(E_u) +2ku_{tt}^{1-k}\sigma_k(E_u)+ u_{tt}^{1-k}\langle T_{k-1}(E_u), -u_{tt}A + u_{tt}(\nabla  u \otimes\nabla u -\dfrac{1}{2}|\nabla u|^2 g) \rangle\\
=& (k+1)u_{tt}^{1-k}\sigma_k(E_u)+u_{tt}^{2-k}\langle T_{k-1}(E_u), -A + \nabla  u \otimes\nabla u -\dfrac{1}{2}|\nabla u|^2 g \rangle.
\end{split}
 \]
where we have used \eqref{Eu1}.
\end{proof}

\begin{rmk}Both the propositions above are derived in \cite{GS-2017} for general $k$. We include the computations here for completeness.
\end{rmk}

We use the operator $D=(\p_t, \nabla)$ to denote the gradient on $\R\times M$, where the space derivative $\phi_k$ denotes the covariant derivative $\nabla _k\phi$. We rewrite \eqref{E3} as,
\begin{equation}\label{E4}
\cL_{F_k}(\phi)= u_{tt}^{-1}f \phi_{tt}+u_{tt}^{1-k}\left\langle T_{k-1}(E_u), P_u(D^2\phi)\right \rangle,
\end{equation}
where
\begin{equation}\label{P1}
P_u(D^2\phi)=u_{tt}\cL_{A_u} \phi-\nabla u_t\otimes \nabla \phi_t-\nabla \phi_t\otimes \nabla u_t+u_{tt}^{-1}\phi_{tt} \nabla u_t\otimes \nabla u_t.
\end{equation}

\begin{prop}\label{quadratic}We have the following,
\begin{equation}\label{E5}
\cL_{F_k}(\phi\psi)=\phi \cL_{F_k}(\psi)+\psi \cL_{F_k}(\phi)+Q_u(D\phi, D\psi)+2u_{tt}^{-1}f \phi_t\psi_t,
\end{equation}
where $Q_u$ is a quadratic form on $D\phi, D\psi$ given by
\begin{equation*}
Q_u(D\phi, D\psi)=u_{tt}^{1-k}\langle T_{k-1}(E_u),u_{tt}\nabla\phi\boxtimes\nabla\psi-\phi_t\nabla u_t\boxtimes \nabla\psi -\psi_t\nabla u_t\boxtimes \nabla\phi +2u_{tt}^{-1}\phi_t\psi_t\nabla u_t\otimes \nabla u_t\rangle.
\end{equation*}
Moreover, we compute
\begin{equation}\label{exponential}\cL_{F_k}(e^\phi)=e^\phi \cL_{F_k}(\phi)+e^{\phi}\left(\frac{1}{2}Q_u(D\phi, D\phi)+u_{tt}^{-1} f\phi_t^2\right)
\end{equation}
An important feature is that $Q_u$ is positive definite in the sense that
\[
Q_u(D\phi, D\phi)\geq 0.
\]
\end{prop}
\begin{proof}
This is a straightforward computation. The main point is that $\cL_{F_k}$ and $P_u$ are second order linear differential operator and the product rule would introduce mixed terms on first derivatives, which lead to the terms $Q_u(D\phi, D\psi)+2u_{tt}^{-1}f \phi_t\psi_t$. Similarly this applies to $e^\phi$. Since $E_u\in \Gamma^+_k$, $T_{k-1}(E_u)$ is positive definite, it follows that
\[
Q_u(D\phi, D\phi)=2u_{tt}^{1-k}\langle T_{k-1}(E_u), Y\otimes Y\rangle\geq 0,
\]
where  $Y=(\sqrt{u_{tt}}) \nabla \phi-\phi_t \nabla u_t(\sqrt{u_{tt}})^{-1} .$ Clearly the positivity of $Q_u$ is simply the consequence of the ellipticity of $F_k$.
\end{proof}

\begin{prop}\label{propelu} We compute, using \eqref{exponential} in  Proposition \ref{quadratic},
\begin{equation}\label{g2}
\cL_{F_k}(e^{-\l u})=\l e^{-\l u} \cL_{F_k}(-u)+\l^2e^{-\l u}\left(\frac{1}{2}Q_u(Du, Du)+u_{tt}^{-1}f u_t^2\right).
\end{equation}
\end{prop}

\begin{prop}We compute,
\begin{equation}\label{E9}
\cL_{F_k}(u_t^2)=2u_t f_t+2fu_{tt}
\end{equation}
\end{prop}
\begin{proof}By \eqref{E5}, we have
\[
\cL_{F_k}(u_t^2)=2u_t\cL_{F_k}(u_t)+Q_u(Du_t, D u_t)+2fu_{tt}
\]
Since taking time derivative has the same effect of taking variation, this gives
\[
\cL_{F_k}(u_t)=\p_t F_k=f_t.
\]
It is clear that $Q_u(Du_t, Du_t)=0$. This completes the computation.
\end{proof}

\begin{prop}\label{gradient}We compute
\begin{equation}\label{C1}
\mathcal{L}_{F_k}(|\nabla u|^2)=2\nabla f \nabla u- 2u_{tt}^{2-k}\langle T_{k-1}(E_u),  \nabla u \nabla A+ Rm(\nabla u,\nabla u)\rangle+Q_u(Du_i,Du_i)+2u_{tt}^{-1}f|\nabla u_t|^2.
\end{equation}
where we denote,
\[
Rm(\nabla u, \nabla u)=R_{ilpk} u_i u_p \p_l\otimes \p_k
\]
\end{prop}
\begin{proof}First we compute, applying \eqref{E5} to $\phi=\psi=u_i$,
\begin{equation}\label{E10}
\cL_{F_k}(|\nabla u|^2)=2u_i\cL_{F_k}(u_i)+Q_u(Du_i, Du_i)+2u_{tt}^{-1}f|\nabla u_t|^2.
\end{equation}
Now we compute
\begin{align}
 \begin{split}
\mathcal{L}_{F_k}(u_i)
=& u_{tt}^{-1}fu_{itt} +  u_{tt}^{1-k}\langle T_{k-1}(E_u),  u_{tt}(\nabla^2 u_i+ \nabla  u \boxtimes \nabla u_i- \langle \nabla u , \nabla u_i \rangle g)\\
&
- \nabla u_t \boxtimes \nabla u_{it} +u_{tt}^{-1}u_{itt} \nabla u_t\otimes \nabla u_t \rangle .
\end{split}
\end{align}
Differentiate equation
$
\sigma_k(E_u)=fu_{tt}^{k-1}
$ we get
\begin{equation}
\langle T_{k-1},(E_u), \nabla_i E_u \rangle=f_iu_{tt}^{k-1} +(k-1)fu_{tt}^{k-2}u_{itt}.
\end{equation}
We compute
\begin{equation}
\nabla_iE_u = u_{tti}A_u + u_{tt}(\nabla_i A + \nabla_i \nabla^2 u + \nabla_i\nabla u \boxtimes \nabla u- \langle \nabla_i\nabla u, \nabla u \rangle g)
-\nabla u_{it}\boxtimes \nabla u_t.
\end{equation}
Note that
$$  \nabla_i \nabla^2 u - \nabla^2 \nabla_i u  = R_{ilpk}u_p \partial_l\otimes\partial_k.$$
So
\begin{equation}\label{E6}
 \begin{split}
\mathcal{L}_{F_k}(u_i)
=& u_{tt}^{-1}fu_{itt} +  u_{tt}^{1-k}\langle T_{k-1}(E_u), \nabla_i E_u- u_{tti} A_u -u_{tt}(\nabla_i A+R_{ilpk}u_p\partial_l\otimes\partial_k)\\&+\dfrac{u_{itt}}{u_{tt}} \nabla  u_t \otimes \nabla u_t\rangle\\
=&ku_{tt}^{-1}fu_{itt} + f_i- u_{tt}^{1-k}\langle T_{k-1}(E_u),  u_{tti}u_{tt}^{-1}E_u +u_{tt}\nabla_i A+u_{tt}R_{ilpk}u_p\partial_l\otimes\partial_k \rangle\\
=&f_i-u_{tt}^{2-k}\langle T_{k-1}(E_u),  \nabla_i A+R_{ilpk}u_p\partial_l\otimes\partial_k \rangle.
\end{split}
 \end{equation}
This completes the computation by combining \eqref{E10} and \eqref{E6}.
\end{proof}

\begin{rmk}The computations above are essentially derived in \cite{GS-2017} for general $k$. We use the quadratic form $Q_u$ to simplify the notations and computations. Of course the positivity of $Q_u$ is equivalent to the fact that $F_k$ is an elliptic operator.
\end{rmk}

Now we prove the estimate for $|\nabla u|^2$.
Combining all the computations above, we have the following estimates,
\begin{lemma}\label{keylemma}
For $\l, b\geq 1$ sufficiently large, we have
\begin{equation}
\cL_{F_k}(e^{-\l u}+bt^2)\geq -C_4 f+ e^{\l }u_{tt}^{2-k}\sigma_1(T_{k-1})\left(\frac{1}{2}|\nabla u|^2-1\right)+ e^{\l }(\sigma_k(A_u)+f u_t^2u_{tt}^{-1})\\
\end{equation}
where $C_4=C_4(\l, |u|_{C^0})$.
\end{lemma}

\begin{proof}
By Proposition \ref{propu}, we get
 \begin{align}
 \begin{split}
\mathcal{L}_{F_k} (- u)=& -(k+1) f +  u_{tt}^{2-k}\langle T_{k-1}(E_u), A - \nabla  u \otimes\nabla u +\frac{1}{2}|\nabla u|^2 g \rangle\\
\ge & -(k+1) f-c_0 u_{tt}^{2-k}\sigma_1(T_{k-1})+ u_{tt}^{2-k}\frac{1}{2} \sigma_1(T_{k-1})|\nabla u|^2- u_{tt}^{2-k}\langle  T_{k-1}(E_u), \nabla u\otimes \nabla u\rangle.
 \end{split}
 \end{align}

We claim that for a constant $C_2\geq 2 u_t^2$,
\[
Q_u(Du, Du)+C_2 u_{tt}^{2-k}\sigma_k(A_u)\geq  u_{tt}^{2-k}\langle T_{k-1}(E_u), \nabla u\otimes \nabla u\rangle
\]
We estimate,
\[
\begin{split}
Q_u(Du, Du)=&\frac{2}{u_{tt}^{k-1}}\langle T_{k-1}(E_u), u_{tt}\nabla u\otimes \nabla u-u_t\nabla u\boxtimes \nabla u_t+u_{tt}^{-1}u_t^2 \nabla u_t\otimes \nabla u_t\rangle\\
= &u_{tt}^{1-k}\langle T_{k-1}(E_u), u_{tt}\nabla u\otimes \nabla u-2u_t\nabla u\boxtimes \nabla u_t+4u_{tt}^{-1}u_t^2 \nabla u_t\otimes \nabla u_t\rangle\\
&+ u_{tt}^{2-k}\langle T_{k-1}(E_u), \nabla u\otimes \nabla u\rangle-2u_{tt}^{-k} u_t^2\langle T_{k-1}(E_u), \nabla u_t\otimes \nabla u_t\rangle\\
=&u_{tt}^{1-k}\langle T_{k-1}(E_u), Y\otimes Y\rangle+ u_{tt}^{2-k}\langle T_{k-1}(E_u), \nabla u\otimes \nabla u\rangle-2u_{tt}^{-k} u_t^2\langle T_{k-1}(E_u), \nabla u_t\otimes \nabla u_t\rangle,
\end{split}
\]
where $Y=\sqrt{u_{tt}}\nabla u-2(\sqrt{u_{tt}})^{-1} u_t \nabla u_t$. The claim follows since
\[
u_{tt}^{2-k}\sigma_k(A_u)-u_{tt}^{-k} \langle T_{k-1}(E_u), \nabla u_t\otimes \nabla u_t\rangle= f u_{tt}^{1-k}\geq 0.
\]
Choose $b\geq (\frac{C_2}{4}+1) \l^2 e^{-\l u}$, then we estimate by Proposition \ref{propelu}
\begin{equation}
\begin{split}
\cL_{F_k}(e^{-\l u}+bt^2)\geq & \l e^{-\l u}\left(-(k+1)f-c_0 u_{tt}^{2-k}\sigma_1(T_{k-1})+ u_{tt}^{2-k}\frac{1}{2} \sigma_1(T_{k-1})|\nabla u|^2- u_{tt}^{2-k}\langle  T_{k-1}(E_u), \nabla u\otimes \nabla u\rangle\right)\\
&+\l ^2 e^{-\l u}\left(\frac{1}{2}Q_u (Du, Du)+fu_t^2 u_{tt}^{-1}\right)+2b\sigma_k(A_u)\\
\geq & -(k+1) \l e^{-\l u} f+\l e^{-\l u}u_{tt}^{2-k}\sigma_1(T_{k-1})\left(\frac{1}{2}|\nabla u|^2-c_0\right)+\l^2 e^{-\l u}(\sigma_k(A_u)+f u_t^2u_{tt}^{-1})\\
&+\left(\frac{\l ^2}{2}-\l\right) e^{-\l u} \langle  T_1(E_u), \nabla u\otimes \nabla u\rangle.
\end{split}
\end{equation}
This completes the proof if $\l$ is sufficiently large.
\end{proof}

\begin{lemma}There exists a uniform constant $C=C(n,k\sup f, \sup |\nabla f^{\frac{1}{k+1}}|, g, |u_0|_{C^1}, |u_1|_{C^1})$ such that
\[
|\nabla u|\leq C.
\]
\end{lemma}

\begin{proof}We take the barrier function
\[
w=|\nabla u|^2+e^{-\l u}+bt^2,
\]
where $\l, b$ are the constants in Lemma \ref{keylemma}.
We compute
\[
\cL_{F_k}(w)=\cL_{F_k}(|\nabla u|^2)+\cL_{F_k}(e^{-\l u}+bt^2).\]
We have, by Proposition \ref{gradient}, that
\begin{equation}\label{eq41}
\cL_{F_k}(|\nabla u|^2)\geq 2\nabla f\nabla u-C_0 u_{tt}^{2-k}\sigma_1(T_{k-1})-C_1u_{tt}^{2-k}\sigma_1(T_{k-1})|\nabla u|^2.
\end{equation}
Suppose $ |\nabla u|$ is large enough, otherwise we have done. Hence by Lemma \ref{keylemma} and \eqref{eq41}, we have
\[
\cL_{F_k}(w)\geq 2\nabla f\nabla u-C_2 f+ku_{tt}^{2-k}\sigma_1(T_{k-1}) |\nabla u|^2+f u_t^2 u_{tt}^{-1}.
\]
If $w$ achieves its maximum on the boundary, then we are already done. Otherwise, suppose  $w$ achieves its maximum at $p=(t, x)\in (0, 1)\times M$. Then $\cL_{F_k}(w)\leq 0$ at $p$.
Hence it follows that (at $p$)
\[
ku_{tt}^{2-k}\sigma_1(T_{k-1}) |\nabla u|^2+f u_t^2 u_{tt}^{-1}\leq 2|\nabla f||\nabla u|+C_2 f
\]
We compute by inequality of arithmetic and geometric mean
\[
ku_{tt}^{2-k}\sigma_1(T_1)|\nabla u|^2+f u_t^2 u_{tt}^{-1}\geq (k+1)\left(\sigma_1(T_{k-1})^k u_{tt}^{(2-k)k-1} f u_{t}^2|\nabla u|^{2k}\right)^{\frac{1}{k+1}}.
\]
By Proposition \ref{prop3.1} and Proposition \ref{MI}
\[\sigma_1^k(T_{k-1})u_{tt}^{(2-k)k-1}=(n-k+1)^k \sigma_{k-1}^k(E_u) u_{tt}^{-(k-1)^2}\geq C(n,k)\sigma_k^{k-1}(E_u) u_{tt}^{-(k-1)^2}=C(n,k)f^{k-1},
\]
it follows that (at $p$)
\[
f^{\frac{k}{k+1}} |\nabla u|^{\frac{2k}{k+1}}\leq |\nabla f| |\nabla u|+C(n,k) f.
\]
This gives the upper bound of $|\nabla u|$ at $p$, and hence the upper bound of $w$. It is not hard to check the dependence of the constants.
\end{proof}

\subsection{$C^2$ estimate}

Now we derive the estimates of second order. The estimates of second order contain the boundary estimates and the interior estimates.

\subsubsection{\bf Boundary estimates}

The boundary is given by two time slices $\{t=0\}\times M$ and $\{t=1\}\times M$.   The tangential-tangential direction, namely $|\nabla^2u|$ is immediate by the boundary data $|\nabla^2 u_0|, |\nabla^2 u_1|$. While the usual ``harder" part of the normal-normal direction  ($u_{tt}$) follows directly from the equation once the tangential-normal direction ($|\nabla u_t|$) is bounded,
\[
u_{tt}\sigma_k(A_u)=\langle T_{k-1}(A_u), \nabla u_t\otimes \nabla u_t\rangle+f.
\]
Note that $\sigma_k(A_u)\geq \delta>0$ at $t=0$ and $t=1$, for some uniform constant $\delta$ depending only on $u_0, u_1$.
Hence one only needs to bound $|\nabla u_t|$ on the boundary. Such a uniform estimate has been obtained by Gursky-Streets in \cite{GS-2017} for the equation for all $1\leq k\leq n$,
\begin{equation}\label{gs}
u_{tt}^{1-k}\sigma_k(E_u)=f
\end{equation}
They stated their results for $E_u^\epsilon=(1+\epsilon)u_{tt}A_u-\nabla u_t\otimes \nabla u_t$ but $\epsilon$ does not play any role in their argument.
We summarize their results as follows,
\begin{thm}[Gursky-Streets \cite{GS-2017}]\label{TGS}If $E_u\in \Gamma^+_k$ and $u$ solves \eqref{gs}. There exists a uniform constant $C_3$, such that
\[
\max_{M\times \{0, 1\}}(u_{tt}+|\nabla^2 u|+|\nabla u_t|)\leq C_3.
\]
\end{thm}

\subsubsection{\bf Interior estimates}

Note that $A_u\in \Gamma^+_k$ implies that $\sigma_1(A_u)>0$.  Given the uniform bound on $|\nabla u|$,
\[
\sigma_1(A_u)=\text{Tr}(A)+\Delta u+\left(1-\frac{n}{2}\right) |\nabla u|^2>0.
\]
This leads to a lower bound of $\Delta u$: there exists a constant $C_2$ such that  $\Delta u+C_2\geq 1. $
Moreover, this gives the equivalence of $\sigma_1(A_u)$ and $\Delta u$ in the sense
\begin{equation}
|\sigma_1(A_u)-\Delta u|\leq C_2.
\end{equation}
We want to derive  upper bound on $u_{tt}$ and $\Delta u+C_2$ (equivalently, the upper bound of $\sigma_1(A_u)$), which will imply the full hessian bound of $u$ since $A_u\in \Gamma^+_k$, and
\[
|A_u|^2=\sigma_1(A_u)^2-2\sigma_2(A_u)\leq \sigma_1(A_u)^2.
\]
The bound on $|\nabla u_t|$ will follow from (\ref{fi}) in Proposition \ref{Lem1} for $i=1$, in the sense that
\[
u_{tt}\sigma_1(A_u)-|\nabla u_t|^2>0.
\]

Gursky-Streets obtained interior $C^2$ estimates for \eqref{gs}, depending on the parameter $\epsilon^{-1}$.
Their computations of $\cL_{F_k}(u_{tt})$ and $\cL_{F_k}(\Delta u)$ in \cite{GS-2017} are very involved.
Here we offer a variant of such computations and this provides significant simplifications.
The complicated nonlinear terms of first order in $A_u$ and the curvature of the background metric will bring extra challenge, not only making the computations much more complicated, but also introducing several nonlinear terms which need extra care.

We need some preparations.
Given a symmetric matrix $R=(r_{ij})$ of $(n+1)\times (n+1)$, we use $r=(r_{ij})$ for the $n\times n$ portion with $ij\neq 0$ and $Y=(r_{01}, \cdots, r_{0n})$.
Note that
\[
F_k(R)=r_{00}\sigma_k(r)-\langle T_{k-1}(r), Y\otimes Y\rangle, \;\text{and}\; G_k(R)=\log F_k(R).
\]
We use the standard notation
\[
{G_k}^{ij}=\frac{\p G_k}{\p r_{ij}}=F_k^{-1} F_k^{ij}, G_k^{ij, lm}=\frac{\p^2 G_k}{\p r_{ij}\p r_{lm}}.
\]
Take the matrix $R$ of the form
\[
R=\begin{pmatrix}u_{tt} & \nabla u_t\\
\nabla u_t& A_u
\end{pmatrix}.
\]
Then we write the equation $F_k(R)=f$ and its equivalent form $G_k(R)=\log f. $
With this notation, we also record the linearization of $F_k(R)$. Given a smooth function $\phi$, we have
\begin{equation}\label{second0}
\cL_{F_k}(\phi)=F_k^{ij}\Phi_{ij}, \text{with}\; \Phi= \begin{pmatrix}\phi_{tt} & \nabla \phi_t\\
\nabla \phi_t& \cL_{A_u}\phi
\end{pmatrix}.
\end{equation}
We record the derivatives of $F_k$.
\begin{prop}\label{F-derivative}We have
\[
G_k^{ij}=F_k^{-1} F_k^{ij}, G_k^{ij, lm}=F_k^{-1}F_k^{ij, lm}-F_k^{-2}F_k^{ij}F_k^{lm}.
\]
We compute, for $i,j\neq 0$,
\begin{equation}\label{F-derivative2}
\begin{split}
&F_k^{00}=\sigma_k(r), F_k^{00, 00}=0,\; F_k^{00, i0}=0, \;\;F_k^{00, ij}=T_{k-1}(r)^{ij}\\
&F_k^{i0}=-\langle T_{k-1}(r), Y\boxtimes e_i \rangle=F_k^{0i} ,\; F_k^{lm}=r_{00}^{2-k}\langle  T_{k-1}(r_{00}r-Y\otimes Y), e_{lm}\rangle.
\end{split}
\end{equation}
\end{prop}
\begin{proof}This is a straightforward computation.
\end{proof}

Now we are ready to compute $\cL_{F_k}(u_{tt})$ and $\cL_{F_k}(\Delta u)$.
\begin{prop}\label{tt} We have the following,
\begin{equation}
\cL_{F_k}(u_{tt})=f_{tt}-f_t^2 f^{-1}-f G_k^{ij, lm} \p_t r_{ij} \p_t r_{lm}-u_{tt}^{2-k}\langle T_{k-1}(E_u), 2\nabla u_t\otimes \nabla u_t-|\nabla u_t|^2 g\rangle.
\end{equation}
\end{prop}
\begin{proof}
We compute
\[
\p_t G_k=G_k^{ij} \p_t r_{ij}=f_t f^{-1}, \p_t^2 (G_k)=G_k^{ij, lm} \p_t r_{ij} \p_t r_{lm}+G_k^{ij}\p^2_t r_{ij}= f_{tt} f^{-1}- (f_tf^{-1})^2.
\]
That is

\begin{equation}\label{second1}
G_k^{ij, lm} \p_t r_{ij} \p_t r_{lm}+F_k^{-1}F_k^{ij}\p^2_t r_{ij}=f_{tt} f^{-1}- (f_tf^{-1})^2.
\end{equation}
Now we consider
\[
(\p^2_t r_{ij})=\p^2_t R=\begin{pmatrix}\p^2_tu_{tt} & \p^2_t\nabla u_t\\
\p^2_t\nabla u_t& \p^2_tA_u
\end{pmatrix}.
\]
The main point is that $A_u$, hence $R$ is not linear on $u$.  We compute
\[
\begin{split}
\p^2_t A_u=& \nabla^2 u_{tt}+\nabla u_{tt}\boxtimes \nabla u-(\nabla u_{tt}, \nabla u) g+2\nabla u_t\otimes \nabla u_t-|\nabla u_t|^2 g\\
=& \cL_{A_u} u_{tt}+2\nabla u_t\otimes \nabla u_t-|\nabla u_t|^2 g.
\end{split}\]
Denote $\cR=2\nabla u_t\otimes \nabla u_t-|\nabla u_t|^2 g$ and this is the term coming from the nonlinearity of $A_u$.
Hence we can write, with $\phi=u_{tt}$,
\[
\p^2_t R=\begin{pmatrix} {\phi}_{tt} & \nabla \phi_t\\
\nabla \phi_t& \cL_{A_u}(\phi)+\cR
\end{pmatrix}.
\]
By \eqref{second0} and \eqref{second1}, we get that
\[
G_k^{ij, lm} \p_t r_{ij} \p_t r_{lm}+F_k^{-1}\cL_{F_k}(u_{tt})+F_k^{-1}F_k^{ij}\cR_{ij}=f_{tt} f^{-1}-(f_tf^{-1})^2,
\]
where we use the notation $\cR_{i0}=0$, for $i=0, 1, \cdots, n$. We claim that
\[
F_k^{ij}\cR_{ij}=u_{tt}^{2-k}T_{k-1}(E_u)^{ij}\cR_{ij}=u_{tt}^{2-k}\langle T_{k-1}(E_u), \cR\rangle.
\]
But this is straightforward since $F_k=u_{tt}^{1-k}\sigma_k(E_u)$,
\[
F_k^{ij}=u_{tt}^{2-k}\langle T_{k-1}(E_u), e_{ij}\rangle, ij\neq 0.
\] This completes the proof.
\end{proof}

Next we compute $\cL_{F_k}(\Delta u)$.
Note that
$$ F_k(u_{tt},A_u,\nabla u_t)=u_{tt}^{1-k}\sigma_k(E_u)=u_{tt}\sigma_k(A_u)-\langle T_{k-1}(A_u), \nabla u_t\otimes \nabla u_t \rangle.$$

\begin{prop}\label{Pro1}
\begin{align}\label{delta}
\mathcal{L}_{F_k}(\triangle u)= -fG_k^{ij,lm}\nabla_pr_{ij}\nabla_pr_{lm} + \triangle f -|\nabla f|^2f^{-1}-F_k^{ij}\mathcal{R}_{1,ij},
\end{align}
where $\cR_1$ is given in \eqref{r1} and \eqref{r2}. We have the following,
 \begin{equation}\label{eq3.49}
 F_k^{ij}\mathcal{R}_{1,ij}=-2u_{tt}^{1-k}\langle T_{k-1}(E_u), Ric(\nabla u_t,\cdot)\boxtimes \nabla u_t \rangle + u_{tt}^{2-k}\langle T_{k-1}(E_u),S_1 +Rm\ast \nabla^2 u + S_0 \rangle
  \end{equation}
with
 \begin{equation}
 S_1=2 \sum_p \nabla\nabla_p u\nabla \otimes \nabla\nabla_p u- |\nabla^2 u|^2g.
  \end{equation}
  where $\cS_0$ is a uniformly bounded term (matrix) and $Rm*\nabla^2 u$ denotes two terms of contraction of curvature with $\nabla^2u$ (which we do not need precise expression).
  \begin{proof}
We compute
\begin{equation}\label{r3}
\Delta G_k(R)=G_k^{ij, kl} \nabla_p r_{ij}\nabla_p r_{kl}+F_k^{-1} F_k^{ij} \Delta r_{ij}=\Delta f f^{-1}-|\nabla f|^2 f^{-2}.
\end{equation}
Now we compute
\[
\left(\Delta r_{ij}\right)=\Delta R= \begin{pmatrix}\Delta u_{tt} & \Delta \nabla u_t\\
\Delta \nabla u_t& \Delta A_u
\end{pmatrix}.
\]
Recall $A_u=A+\nabla^2 u+\nabla u\otimes \nabla u-|\nabla u|^2 g/2$ and now we compute $\Delta A_u$.
We need several Bochner-Weitzenbock formula as follows,
\[
\begin{split}
&\Delta \nabla u_t=\nabla \Delta u_t+Ric(\nabla u_t, \cdot), \Delta \nabla^2 u=\nabla^2 \Delta u+Rm*\nabla^2 u+\nabla Rm *\nabla u.\\
&\Delta (\nabla u\otimes \nabla u)=\nabla \Delta u\boxtimes \nabla u+Ric(\nabla u, \cdot)\boxtimes \nabla u+2\nabla \nabla_p u\otimes \nabla \nabla_p u.\\
&\Delta \left(\frac{1}{2}|\nabla u|^2\right)=|\nabla^2 u|^2+Ric(\nabla u, \nabla u)+\langle \nabla \Delta u, \nabla u\rangle.
\end{split}
\]
We use $Rm*\nabla^2 u+\nabla Rm *\nabla u$ to denote contraction of terms which we do not need precise expression.
We can then compute
\[
\Delta R= \begin{pmatrix}(\Delta u)_{tt} & \nabla \Delta u_t+Ric(\nabla u_t, \cdot)\\
\nabla \Delta u_t+Ric(\nabla u_t, \cdot) & \cL_{A_u}(\Delta u)+\cS
\end{pmatrix},
\]
where $\cS$ is the remaining matrix of the form
\begin{equation}\label{r2}
\begin{split}
\cS=&Ric(\nabla u, \cdot)\boxtimes \nabla u+2\nabla \nabla_p u\otimes \nabla \nabla_p u-(|\nabla^2 u|^2+Ric(\nabla u, \nabla u))g\\
&+ \Delta A+Rm*\nabla^2 u+\nabla Rm *\nabla u.
\end{split}
\end{equation}
Denote
\begin{equation}\label{r1}
\cR_1=\begin{pmatrix}0 & Ric(\nabla u_t, \cdot)\\
Ric(\nabla u_t, \cdot) & \cS
\end{pmatrix}.
\end{equation}
Then we can write
\[
\Delta R=\cR_1+\begin{pmatrix}(\Delta u)_{tt} & \nabla \Delta u_t\\
\nabla \Delta u_t & \cL_{A_u}(\Delta u)
\end{pmatrix}.
\]
It then follows that
\[
F_k^{ij}\Delta r_{ij}=\cL_{F_k}(\Delta u)+F_k^{ij}\cR_{1, ij}.
\]
Together with \eqref{r3} this completes the proof of \eqref{delta}. The computation of $F_k^{ij}\cR_{1, ij}$ in \eqref{eq3.49} is straightforward, noting that
\[
F_k^{i0}=-u_{tt}^{1-k}\langle T_{k-1}(E_u), \nabla u_t\boxtimes e_i\rangle, i\neq 0.
\]
\end{proof}

To estimate $u_{tt}$ and $\Delta u$, we need the following Lemma.
\begin{lemma}\label{Lem2}
Let $\phi$ be any smooth function. For $n \ge 2k$,
 \begin{align}\label{qd}
\langle T_{k-1}(E_u), \frac{|\nabla\phi|^2}{2}g-\nabla\phi\otimes\nabla\phi\rangle \ge
\frac{(n-2k)(n-k+1)}{2n} \sigma_{k-1}(E_u)|\nabla\phi|^2
\end{align}
\end{lemma}

\begin{proof}
Diagonalize $E_u$ with eigenvalues $\lambda=(\lambda_1,...,\lambda_n)$, so
\begin{align}
 \begin{split}
&\langle T_{k-1}(E_u), \dfrac{|\nabla\phi|^2}{2}g-\nabla\phi\otimes\nabla\phi\rangle \\
=&\dfrac{1}{2}\sigma_1(T_{k-1}(\lambda)|\nabla\phi|^2-\sigma_{k-1}(\lambda|i)|\nabla_i \phi|^2\\
=&\dfrac{n-k+1}{2}\sigma_{k-1}(\lambda)|\nabla\phi|^2-\sigma_{k-1}(\lambda|i)|\nabla_i \phi|^2,
\end{split}
 \end{align}
then \eqref{qd} is equivalent to prove
 \begin{align}
 \begin{split}
&\left[\dfrac{n-k+1}{2}-\dfrac{(n-2k)(n-k+1)}{2n}\right ]\sigma_{k-1}(\lambda)|\nabla\phi|^2-\sigma_{k-1}(\lambda|i)|\nabla_i \phi|^2\\
= &\dfrac{k(n-k+1)}{n}[\sigma_{k-1}(\lambda|i)+\lambda_i\sigma_{k-2}(\lambda|i)]|\nabla_i \phi|^2-\sigma_{k-1}(\lambda|i)|\nabla_i \phi|^2\\
= & \left [(k-1)(n-k)\sigma_{k-1}(\lambda|i)+k(n-k+1)\dfrac{\sigma_{k}(\lambda)-\sigma_{k}(\lambda|i)}{\sigma_{k-1}(\lambda|i)}\sigma_{k-2}(\lambda|i)\right ]\dfrac{|\nabla_i \phi|^2}{n}\\
\ge &  [(k-1)(n-k)\sigma_{k-1}^2(\lambda|i)-k(n-k+1)\sigma_{k}(\lambda|i)\sigma_{k-2}(\lambda|i)]\dfrac{|\nabla_i \phi|^2}{n\sigma_{k-1}(\lambda|i)}\\
\ge & 0.
\end{split}
 \end{align}
 The last inequality holds from Newton's inequality \eqref{eqN1} noting that $\lambda | i \in \mathbb{R}^{n-1}$.
\end{proof}
The interior estimate of $u_{tt}$ now becomes immediate ($n \ge 2k)$.
\begin{lemma}For $n\geq 2k$, there exists a constant $C_3$ such that
\[
u_{tt}\leq C_3.
\]
\end{lemma}
\begin{proof}By the concavity of $G$, Lemma \ref{Lem2} and Proposition \ref{tt}, we have
\[
\cL_F(u_{tt})\geq f_{tt}-f_t^2 f^{-1}.
\]
It then follows that, using \eqref{E9},
\[
\cL_F(u_{tt}+u_t^2)\geq 2u_t f_t+2f u_{tt}+f_{tt}-f_t^2 f^{-1}.
\]
If $u_{tt}+u_t^2$ achieves its maximum on the boundary, then by Theorem \ref{TGS} and Proposition~\ref{pro3.8} we are done. Otherwise at the maximum point of $u_{tt}+u_t^2$, we have
\[
2u_t f+2f u_{tt}+f_{tt}-f_t^2 f^{-1}\leq 0
\]
This is sufficient to bound $u_{tt}$ by a uniform constant $C_3$, where $C_3$ depends on the boundary estimate of $u_{tt}$ and $-f_{tt}f^{-1}, |f_t|f^{-1}$ in addition.
\end{proof}
\end{prop}

\begin{prop}\label{Pro2}
For $n \ge 2k+1$, there exists a uniform constant $C$ such that
$$\triangle u \le C.$$
\end{prop}
\begin{proof}
By Lemma \ref{Lem1}, $f_1 >0 $ then \eqref{fi} yields
\begin{align}\label{}
u_{tt}\sigma_1(A_u)\ge |\nabla u_t|^2.
 \end{align}
So in Proposition \ref{Pro1},
\begin{align}
 \begin{split}\label{eqr}
 &2u_{tt}^{1-k}\langle T_{k-1}(E_u), Ric(\nabla u_t,\cdot)\boxtimes \nabla u_t \rangle\\~
 \ge& -C_1u_{tt}^{1-k} \sigma_1(T_{k-1}(E_u))|\nabla u_t|^2\\
 \ge & -C_1 u_{tt}^{2-k}\sigma_1(T_{k-1}(E_u))\sigma_1(A_u)
\end{split}
 \end{align}
and
\begin{align}\label{eqs0}
-u_{tt}^{2-k}\langle T_{k-1}(E_u),Rm\ast \nabla^2 u + S_0 \rangle
\ge -C_2u_{tt}^{2-k}|T_{k-1}(E_u)|(|\nabla^2 u|+1)
\end{align}
In Proposition \ref{Pro1}, combining \eqref{eqr}, \eqref{eqs0} and Lemma \ref{Lem2}, together with the concavity of $G_k$,
\begin{align}
 \begin{split}\label{lflu}
&\mathcal{L}_{F_k}(\triangle u)\\
\ge&  \triangle f -|\nabla f|^2f^{-1} -C_1 u_{tt}^{2-k}\sigma_1(T_{k-1}(E_u))\sigma_1(A_u)\\
&-C_2u_{tt}^{2-k}|T_{k-1}(E_u)|(|\nabla^2 u|+1)+\dfrac{(n-2k)(n-k+1)}{2n}u_{tt}^{2-k} \sigma_{k-1}(E_u)|\nabla^2 u|^2\\
\ge &  \triangle f -|\nabla f|^2f^{-1} -C_3u_{tt}^{2-k}\sigma_{k-1}(E_u)(|\nabla^2 u|+1)\\
&+\dfrac{(n-2k)(n-k+1)}{2n}u_{tt}^{2-k} \sigma_{k-1}(E_u)|\nabla^2 u|^2.
\end{split}
 \end{align}
The last inequality holds since $|T_{k-1}(E_u)| \le \sigma_1(T_{k-1}(E_u))=(n-k+1)\sigma_{k-1}(E_u)$.
Suppose $\triangle u$ obtains its maximum at an interior point $p$ (otherwise done), assume $|\nabla^2 u(p)|$ is sufficiently large(otherwise done), then at $p$ we get
$$ C_4u_{tt}^{2-k} \sigma_{k-1}(E_u)|\nabla^2 u|^2 \le -\triangle f +|\nabla f|^2f^{-1}.$$
By Lemma \ref{Lem1} for $i=k$, it means that
$$ C_4\sigma_{k}(A_u)^{-1}\sigma_{k-1}(A_u)|\nabla^2 u|^2 \le -f^{-1}\triangle f +|\nabla f|^2f^{-2}.$$
This is sufficient to get a uniform upper bound of $\triangle u$.
\end{proof}

The estimates of $\Delta u$ (for $n\geq 2k+1$) is rather straightforward given the strictly lower bound of the quadratic form in Lemma \ref{Lem2}. When $n=2k$, such a positivity is too weak and
the interior estimate of $\Delta u$ is subtle. A key technical ingredient is the following,
\begin{lemma}\label{Pro4}
\begin{align}
 &K\sigma_k(A_u)+\sum_iQ_u(Du_i,Du_i) \ge \sum_i \varepsilon \langle T_{k-1}(E_u),  \nabla\nabla_i u \otimes \nabla\nabla_i u \rangle u_{tt}^{2-k}.
 \end{align}
 \end{lemma}
 \begin{proof}
 Recall
 $$\sum_i Q_u( Du_i,Du_i )=\sum_i 2u_{tt}^{1-k}\langle T_{k-1}(E_u), u_{tt}\nabla\nabla_i u\otimes\nabla\nabla_i u-\nabla_iu_t\nabla u_t\boxtimes\nabla\nabla_i u
 +\frac{|\nabla_i u_t|^2}{u_{tt}}\nabla u_t\otimes\nabla u_t \rangle.$$
 And we get
 \begin{align}
 \begin{split}
 K\sigma_k(A_u)\ge&\frac{|\nabla u_t|^2}{u_{tt}}u_{tt}^{-1}\langle T_{k-1}(A_u),\nabla u_t\otimes \nabla u_t \rangle\\
 \ge & 2u_{tt}^{1-k}\sum_i\langle T_{k-1}(E_u),\epsilon_0 \frac{|\nabla_i u_t|^2}{u_{tt}}\nabla u_t\otimes \nabla u_t \rangle
 \end{split}
 \end{align}
 for some uniformly positive constant $\epsilon_0 \le 1$ such that $u_{tt}^{-1} \ge 2 \epsilon_0$. It follows that
  \begin{align*}
& K\sigma_k(A_u)+\sum_iQ_u(Du_i,Du_i)\\
\ge&
  2u_{tt}^{1-k}\sum_i\langle T_{k-1}(E_u), u_{tt}\nabla\nabla_i u\otimes\nabla\nabla_i u-\nabla_iu_t\nabla u_t\boxtimes\nabla\nabla_i u
 +(1+\epsilon_0)\frac{|\nabla_i u_t|^2}{u_{tt}}\nabla u_t\otimes\nabla u_t \rangle\\
 \ge & \sum_i\varepsilon \langle T_{k-1}(E_u),  \nabla\nabla_i u \otimes \nabla\nabla_i u \rangle u_{tt}^{2-k}.
 \end{align*}
\end{proof}
\begin{prop}\label{Pro5}
$n=2k$, there exists a uniform constant $C$ such that
\begin{equation}\label{lfu}
   \triangle u \le C.
\end{equation}
\end{prop}
\begin{proof}
Consider test function
$$ v=\triangle u + \frac{K}{2} t^2 + |\nabla u|^2
,$$
with
$$ K= \max_{p\in [0,1]\times M} \frac{|\nabla u_t|^2}{u_{tt}}.$$
In Proposition \ref{Pro1}, combining \eqref{eqr}, \eqref{eqs0}, together with the concavity of $G_k$, one get
 \begin{align}\label{3.64}
 \begin{split}
\mathcal{L}_{F_k}(\triangle u)
\ge &  \triangle f -|\nabla f|^2f^{-1} -C_1u_{tt}^{2-k}\sigma_{k-1}(E_u)(|\nabla^2 u|+1)\\
&-\langle T_{k-1}(E_u), S_1 \rangle u_{tt}^{2-k}.
 \end{split}
 \end{align}
Furthermore \eqref{t} and Proposition \ref{gradient} give that
\begin{align}
\mathcal{L}_{F_k} (\frac{K}{2}t^2+|\nabla u|^2)\ge K\sigma_k(A_u)+\sum_iQ_u(Du_i,Du_i)+\frac{2f}{u_{tt}}|\nabla u_t|^2-C_2|\nabla f|-C_3u_{tt}^{2-k}\sigma_{k-1}(E_u).
\end{align}
We note by Lemma \ref{Pro4}
 \begin{align}\label{3.67}
 \begin{split}
 &K\sigma_k(A_u)+\sum_iQ_u(Du_i,Du_i)-\langle T_{k-1}(E_u), S_1 \rangle u_{tt}^{2-k}\\
 \ge &\langle T_{k-1}(E_u), |\nabla^2 u|^2g-(2-\varepsilon)\sum_i \nabla\nabla_i u \otimes \nabla\nabla_i u \rangle u_{tt}^{2-k}\\
 = & \langle T_{k-1}(E_u), (2-\varepsilon)(\frac{1}{2}|\nabla^2 u|^2g-\sum_i \nabla\nabla_i u \otimes \nabla\nabla_i u )+ \frac{\varepsilon}{2}|\nabla^2 u|^2g \rangle u_{tt}^{2-k}\\
 \ge & \frac{\varepsilon}{2} \langle T_{k-1}(E_u),|\nabla^2 u|^2g \rangle u_{tt}^{2-k}\\
 =&  \frac{(n-k+1)\varepsilon}{2} \sigma_{k-1}(E_u)|\nabla^2 u|^2 u_{tt}^{2-k}.
 \end{split}
 \end{align}
Combining above inequalities \eqref{3.64}-\eqref{3.67} we obtain
 \begin{align}
 \begin{split}
\mathcal{L}_{F_k}(v)
\ge & \frac{(n-k+1)\varepsilon}{2} \sigma_{k-1}(E_u)|\nabla^2 u|^2 u_{tt}^{2-k}-C_1u_{tt}^{2-k}\sigma_{k-1}(E_u)|\nabla^2 u|
-(C_1
+C_3)u_{tt}^{2-k}\sigma_{k-1}(E_u)\\
& +\triangle f-|\nabla f|^2f^{-1} 
+\frac{2f}{u_{tt}}|\nabla u_t|^2-C_2|\nabla f|.
 \end{split}
 \end{align}
If $v$ achieves its maximum on the boundary, then we are done. Otherwise $v$ obtains its maximum at an interior point $p$, it follows that
\begin{align}\label{3.68}
 \begin{split}
 &\frac{(n-k+1)\varepsilon}{2} \sigma_{k-1}(E_u)|\nabla^2 u|^2 u_{tt}^{2-k}-C_1u_{tt}^{2-k}\sigma_{k-1}(E_u)|\nabla^2 u|
-(C_1
+C_3)u_{tt}^{2-k}\sigma_{k-1}(E_u)\\
& +\triangle f-|\nabla f|^2f^{-1} 
+\frac{2f}{u_{tt}}|\nabla u_t|^2-C_2|\nabla f| \le 0.
 \end{split}
 \end{align}
 We claim that this is sufficient to bound $|\nabla^2 u|$  at $p$,
 \begin{equation}\label{3.69}
   |\nabla^2 u|(p)\le C_3.
 \end{equation}
We can assume $|\nabla^2 u|(p)\ge 2C_1 \epsilon^{-1}$ (otherwise done), then
\begin{equation}
\frac{\epsilon}{2} \sigma_{k-1}(E_u)|\nabla^2 u|^2 u_{tt}^{2-k}-C_1u_{tt}^{2-k}\sigma_{k-1}(E_u)|\nabla^2 u| \ge 0.
\end{equation}
And by \eqref{3.68} we have
\begin{align*}
\epsilon \sigma_{k-1}(E_u)|\nabla^2 u|^2 u_{tt}^{2-k} \le C_3( f+u_{tt}^{2-k}\sigma_{k-1}(E_u)).
\end{align*}
By Proposition \ref{Lem1} for $i=k$, we get at $p$
$$|\nabla^2 u|^2 \le C_3 \epsilon^{-1}(\frac{\sigma_k(A_u)}{\sigma_{k-1}(A_u)}+1).$$
And
$$|\nabla^2 u|(p) \le C.$$
This establishes claim \eqref{3.69}. Obviously we have $\Delta u(p)\leq n |\nabla^2 u|(p)$. Since $v\leq v(p)$, we have obtained
 \[
 \Delta u\leq v\leq v(p)\leq C_3+\frac{K}{2}.
 \]
In other words, we have
\[
\sup \Delta u\leq C_3 +\frac{K}{2}.
\]
 Note that $\Delta u-C_2\leq \sigma_1(A_u)\leq \Delta u+C_2$, we get
 \[
\sup \sigma_1(A_u)\leq C_3+\frac{K}{2}.
 \]
 We observe that
 \[
 \sigma_1(A_u)-\frac{|\nabla u_t|^2}{u_{tt}}=u_{tt}^{-1}\sigma_1(E_u) >0
 \]
 Hence $K<\sup \sigma_1(A_u)$, and we have proved that
 \[
 \sup \sigma_1(A_u)\leq C_3.
 \]
 This completes the proof of the uniformly upper bound of $\Delta u$ i.e. \eqref{lfu}.
\end{proof}

\subsection{The existence and uniqueness} In this section we prove Theorem~\ref{thmm} and Theorem \ref{thmm2}. Given the estimates in the last section, the proof of Theorem \ref{thmm} is standard and rather straightforward. We shall keep it brief.

\begin{proof}[Proof of Theorem \ref{thmm}]Let $w= (1-t)u_0 +tu_1+at(t-1)$ for $a$ sufficiently large. Then
\[
f_0=F_k(w_{tt}, A_w, \nabla w_t)>0.
\]
For any smooth function $f$, we consider the continuity path for $s\in [0, 1]$
\[
F_k(v^s_{tt}, A_{v^s}, \nabla v^s_t)=(1-s)f_0+sf.
\]
Then for $s=0$, $v^s=w$ solves the equation. The openness follows from the ellipticity and the closedness follows from the a priori estimates.
Hence the equation has a solution for $s=1$.
For any $f>0$, the solution is unique by a straightforward maximum principle argument.
\end{proof}

Next we discuss the  degenerate equation \eqref{GS1} and prove Theorem \ref{thmm2}.
\begin{proof}
We can take $f=s$ for $s\in (0, 1]$. Hence we get a unique smooth solution $u^s$ for
\[
F_k(u^s_{tt}, A_{u^s}, \nabla u^s_t)=s.
\]
The comparison principle implies that $u^s\geq u^{\tilde s}$  for $s\le \tilde s$. Hence $u^s$ converges uniformly to $u$ for $s\rightarrow 0$ with $u\in C^{1, 1}$. In particular $u^s$ converges to $u$ in $C^{1, \alpha}$ for any $\alpha\in (0, 1)$.

Following a standard notion of \emph{viscosity solution}, $u$ solves the degenerate equation
\[
F_k(u_{tt}, A_u, \nabla u_t)=0
\]
in the sense of viscosity solution, see \cite[Lemma 6.1, Remark 6.3]{CIL}. Since $u\in C^{1, 1}$, this implies that
\[
F_k(u_{tt}, A_u, \nabla u_t)=0
\]
holds in $L^\infty$ sense; and in particular $u$ is a strong solution. We complete the proof given the uniqueness proved below.
\end{proof}

The uniqueness problem for degenerate elliptic equations can be subtle.
Our argument is an adaption of the argument used in Guan-Zhang \cite{GZ}.
First we define precisely the notion of \emph{admissible solution} for degenerate equations.

\begin{defn}
A $C^{1, 1}$ function $u$ is called an admissible solution of
\[
F_k(u_{tt}, A_u, \nabla u_t)=0,
\]
if $A_u\in \overline{\Gamma_k^{+}}$ and  $u_{tt}\geq 0.$
\end{defn}

The key point is the following approximation result.

\begin{lemma}\label{Lun}
Suppose $u$ is a $C^{1, 1}$ admissible function defined on $ M\times [0, 1]$ with \[F_k(u)=F_k(u_{tt},A_u,\nabla u_t)= 0,\] defined in \eqref{eqf1}. For any $\delta>0$, there is an admissible function $u_\delta \in C^\infty(M \times [0,1])$ such that $0<F_k(u_\delta) \le \delta$ and $||u-u_\delta||_{C^0(M \times [0,1])} \le \delta $.
\end{lemma}
\begin{proof}For $\epsilon\in {(0, 1)}$, let $v=(1-\epsilon)u$. First we show that $A_v\in \Gamma_k^+$ (in $L^\infty$ sense since $u$ might not be a $C^2$ function).
We have
\[
A_v=A+(1-\epsilon)\nabla^2 u+(1-\epsilon)^2\nabla u\otimes \nabla u-\frac{(1-\epsilon)^2}{2}|\nabla u|^2 g.
\]
Hence we compute
\[
A_v=(1-\epsilon)A_u+\epsilon \left(A+\frac{(1-\epsilon)}{2}|\nabla u|^2 g-(1-\epsilon)\nabla u\otimes \nabla u\right).
\]
It is important to notice that for $2k\leq n$, the $n\times n$ matrix $\text{diag}(-1, 1, \cdots, 1)$ is in $\Gamma_k^{+}$ (this is not the case if we drop the assumption $2k\leq n$. For example, $\text{diag}(-1, 1, 1)$ is not in $\Gamma_2^+$ with $k=2, n=3$). This implies that
\[
\frac{1}{2}|\nabla u|^2 g-\nabla u\otimes \nabla u\in \overline{\Gamma_k^+}.
\]
Hence this follows that
\[
B:=A+\frac{(1-\epsilon)}{2}|\nabla u|^2 g-(1-\epsilon)\nabla u\otimes \nabla u\in \Gamma_k^+.
\]
{So $A_v\in \Gamma_k^+$. Moreover}
By Theorem~\ref{thm2.1},  we have
\[
F_k^{\frac{1}{k+1}}(v_{tt}, A_v, \nabla v_t)\geq (1-\epsilon)F_k^{\frac{1}{k+1}}(u_{tt}, A_u, \nabla u_t)+\epsilon F_k^{\frac{1}{k+1}}(0, B, 0) \geq 0.
\]
Denote $w=v+\epsilon t^2$. Then
\[
F_k(w_{tt}, A_w, \nabla w_t)>0.
\]
We can then approximate $w$ by a smooth function $u_\delta$ such that
\[
0<F_k(u_\delta)\leq \delta,
\]
and $|u-u_\delta|\leq \delta$ by choosing $\epsilon$ sufficiently small (depending on $\delta$).

\end{proof}

We are ready to establish the uniqueness of the degenerate equation.
\begin{prop}
A $C^{1, 1}$ admissible solution to the degenerate equqation \eqref{GS1} with given boundary data is unique.
\end{prop}
\begin{proof}
Suppose there are two such solutions $u_1,u_2$ with the same boundary data. For any $0< \delta <1$, pick any $0< \delta_1, \delta_2 < \delta$, by Lemma~\ref{Lun}, there exist two smooth functions $v_1, v_2$ such that
$$ F_k(v_i)=f_i,$$
in $M \times [0,1]$, $||v_i-u_i||_{C^2(M \times [0,1])} \le \delta_i$ and $0 < f_i \le \delta_i$ for $i=1,2$. Set $w_1=\delta v_1+(1-\delta) t^2$. A straightforward computation gives that \[F_k(w_1) \ge 2\delta^{k+1}\sigma_k(A)\]
We may choose $\delta_2$ sufficiently small such that $0 < \delta_2 < F_k(w_1) $. The maximum principle implies $max_{M \times [0,1]}(w_1-v_2) \le max_{M \times \{0,1\}}(w_1-v_2)$. Thus
\[max_{M \times [0,1]}(u_1-u_2) \le max_{M \times \{0,1\}}(u_1-u_2) +C \delta=C\delta,\]
where constant $C$ depends only on $C^0$ norm of $u_1,u_2$. Interchange the role of $u_1$ and $u_2$, we have
$$max_{M \times [0,1]}|u_1-u_2| \le C\delta.$$
Since $0 < \delta <1$ is arbitrary, we conclude that $u_1=u_2$.
\end{proof}

\section{Appendix}

In this section we introduce the geometric structure briefly related to the Gursky-Streets equation.
It was noted \cite{GS-2018} that when $n=2k\geq 6$ and $M$ is locally conformally flat, the theory is parallel to the case $n=4=2k$, studied extensively in \cite{GS-2017} and \cite{He18}. We include the discussions for completeness.

Let $(M^n, g)$ be a compact Riemannian manifold of dimension $n$, $n\geq 3$. Given a  conformal metric $g_u=e^{-2u}g$.
Recall that  the Schouten tensor is given by
\[
A_u=A(g_u)=A+\nabla^2 u+\nabla u\otimes \nabla u-\frac{1}{2}|\nabla u|^2 g.
\]
Let $g_u=e^{-2u(t)}g$ be a one-parameter family of conformal metrics.
Then a simple computation gives that
\[
\dfrac{\partial}{\partial t}(g_u^{-1}A_u)^j_i=2(\dfrac{\partial u}{\partial t})(g_u^{-1}A_u)^j_i+ (\nabla^2_u \dfrac{\partial u}{\partial t})^j_i,
\]
where the Hessian is with respect to $g_u$. A direct calculation \cite{Reilly} yields
\begin{equation}\label{parsigma}
\dfrac{\partial}{\partial t} \sigma_k(g_u^{-1}A_u)=\langle T_{k-1}(g_u^{-1}A_u), \nabla^2_u \dfrac{\partial u}{\partial t}\rangle_{g_u}+2k\dfrac{\partial u}{\partial t}\sigma_k(g_u^{-1}A_u),
\end{equation}
where $T_{k-1}$ is the Newton transform. Since  the Newton transform is a (1,1)-tensor, for the pairing in \eqref{parsigma} we lower an index of $T_{k-1}(g_u^{-1}A_u)$ and view it as $(0,2)$-tensor, and use the inner product induced by $g_u$. For example, if $n=4$ and $k=2$,
\[
T_1(g_u^{-1}A_u)=-A_u+\sigma_1(g_u^{-1}A_u)g_u
\]
\eqref{parsigma} yields
\begin{equation}\label{par}
\dfrac{\partial}{\partial t} [\sigma_k(g_u^{-1}A_u)\;dV_u]=\langle T_{k-1}(g_u^{-1}A_u), \nabla^2_u \dfrac{\partial u}{\partial t}\rangle_{g_u}\;dV_u+(n-2k)\dfrac{\partial u}{\partial t}\sigma_k(g_u^{-1}A_u)\;dV_u.
\end{equation}

\begin{defn}
For $\phi,\psi \in C^\infty(M)$ we define the $\sigma_k$ metric as following,
\begin{equation}\label{gsmetrick}
\langle \psi, \phi\rangle_{u}=\int_M \phi \psi \sigma_k (g_u^{-1}A_u) dV_u.
\end{equation}
\end{defn}

We will need a key property as following.
\begin{lemma}\label{lecon}
If $k=2$ or if the manifold is locally conformally flat , then $T_{k-1}(g^{-1}A)$ is divergence-free.
\end{lemma}
A straightforward computation gives the following,
\begin{lemma}
If $n=4=2k$ or if $n=2k$ and $M^n$ is conformally flat, the geodesic equation is given by
\begin{equation}\label{es}
u_{tt}\sigma_k(A_u)-\langle T_{k-1}(A_u),\nabla u_t\otimes \nabla u_t\rangle=0.
\end{equation}
\end{lemma}

Here we introduce Andrew's Poincar$\acute{e}$ inequality, see \cite{An}.
\begin{prop}
Let $(M^n,g)$ be a closed Riemannian manifold with positive Ricci curvature. Given $\phi \in C^\infty(M)$ such that $\int_M \phi\; dV=0$, then
\begin{equation}
\int_M (Ric^{-1})^{ij}\nabla_i\phi\nabla_j\phi \;dV \ge \dfrac{n}{n-1}\int_M \phi^2 \;dV,
\end{equation}
with equality if and only if $\phi \equiv  0$ or  $(M^n, g)$ is isometric to the round sphere.
\end{prop}

\begin{cor}\label{coran}
Let $(M^{n},g)$ be a closed Riemannian manifold such that $A_g \in \Gamma^+_k$ with $n=2k$. Given $\phi \in C^\infty(M)$ and let $\underline{\phi}=V_g^{-1}\int_M \phi\; dV_g$, one has
\begin{equation}\label{eqp}
\int_M \dfrac{T_{k-1}(A_g)^{ij}}{\sigma_k(A_g)}\nabla_i\phi\nabla_j\phi \;dV_g \ge n \int_M (\phi-\underline{\phi})^2 dV_g.
\end{equation}
with equality if and only if $\phi$ is a constant or  $(M^{n},g)$ is isometric to the round sphere.
\end{cor}
\begin{proof}
We assume that $\int_M \phi\; dV_g=0$. By Andrew's Poincar$\acute{e}$ inequality, to show the claim \eqref{eqp} it suffices to show
\begin{equation}
\int_M \dfrac{T_{k-1}(A_g)^{ij}}{\sigma_k(A_g)}\nabla_i\phi\nabla_j\phi \;dV_g \ge (n-1)\int_M (Ric^{-1})^{ij}\nabla_i\phi\nabla_j\phi \;dV_g.
\end{equation}
That is to show the matrix inequality
\begin{equation}
\dfrac{T_{k-1}(A_g)}{\sigma_k(A_g)}\ge (n-1)Ric^{-1}.
\end{equation}
Since $Ric$ and $T_{k-1}(A)$ commute, we only need to show that
\begin{equation}
T_{k-1}(A)\circ Ric \ge (n-1)\sigma_k(A).
\end{equation}
Note that $Ric=(n-2)A+\sigma_1(g^{-1}A)g$, we choose coordinate at a point $p \in M$ such that $A$ is diagonal, then $Ric$ and $T_{k-1}$ is also diagonal. Assume the eigenvalues of $A$ are $\lambda_i, i=1,...,n$. Using Proposition \ref{prop3.1} and Proposition \ref{NMI}, we compute
\begin{align*}
T_{k-1}\circ Ric (\lambda) &= (n-2)\sigma_{k-1}(\lambda |i)\lambda_i+\sigma_{k-1}(\lambda |i)\sigma_1(\lambda)\\
& =(n-1)\sigma_{k-1}(\lambda |i)\lambda_i+\sigma_{k-1}(\lambda |i)\sigma_1(\lambda|i)\\
& \ge (n-1)\sigma_{k-1}(\lambda |i)\lambda_i+\dfrac{k(n-1)}{n-k}\sigma_{k}(\lambda |i)\\
& =(n-1)\sigma_k(\lambda).
\end{align*}
\end{proof}

\begin{rmk}We should mention that Guan, Viaclovsky and Wang \cite{GVW} proved that when $2k\geq n$, $Ric$ is positive definite if $g\in \Gamma_k^{+}$. In particular, they proved that if $M^n$ is locally conformally flat and $g\in \Gamma_k^{+}$ for $2k\geq n$, then $M^n$ is conformal to the round sphere $S^n$.
\end{rmk}

Set $\sigma=\int_M \sigma_k(g_u^{-1}A_u)\; dV_u$, and $\overline{\sigma}=V_u^{-1}\sigma$. Let $M^{2k}$ be locally conformally flat. Brendle and Viaclovsky \cite{BV04} introduced a functional which integrates the 1-form
\[
\alpha_g(v)=\int_M v(\bar\sigma-\sigma_k(A_g))dV_g
\]
Define the $F$-functional as $dF=\alpha$.  When $M^n$ is locally conformally flat and $2k=n$, we can get
\begin{prop}
The functional $F$ is formally geodesically convex.
\end{prop}
\begin{proof}It follows from \cite{BV04} that, for a path of conformal metrics $u=u(t)$,
\begin{equation}
\dfrac{d}{dt}F[u]=\int_M u_t[-\sigma_k(g_u^{-1}A_u)+\overline{\sigma}] \; dV_u.
\end{equation}
Differentiating again along the geodesic path we have
\begin{equation}\label{f-convexity}
\begin{split}
\dfrac{d^2}{dt^2}F[u]=&\dfrac{d}{dt} \int_M u_t[-\sigma_k(g_u^{-1}A_u)+\overline{\sigma}] \; dV_u\\
=&-\int_M [u_{tt}\sigma_k(g_u^{-1}A_u)+ u_t\langle T_{k-1}(g_u^{-1}A_u), \nabla^2_u u_t \rangle_{g_u}]\;dV_u\\
&+\sigma\int_M [u_{tt}V_u^{-1}+V_u^{-2}u_t(\int_M nu_t\;dV_u)-nV_u^{-1}u_t^2]\;dV_u\\
=&-\int_M [u_{tt}\sigma_k(g_u^{-1}A_u)-\langle T_{k-1}(g_u^{-1}A_u), \nabla u_t\otimes \nabla u_t \rangle_{g_u}]\;dV_u\\
&+ \sigma V_u^{-1} \left[\int_M\dfrac{\langle T_{k-1}(g_u^{-1}A_u), \nabla u_t\otimes \nabla u_t \rangle_{g_u}}{\sigma_k(g_u^{-1}A_u)}-n(u_t-\underline{u_t})^2\right]\;dV_u\\
\ge & 0,
\end{split}
\end{equation}
where the last line follows from Corollary \ref{coran}.
\end{proof}

We have proved that the Dirichlet problem
\[
u_{tt}\sigma_k(A_u)-\langle T_{k-1}(A_u),  \nabla u_t\otimes \nabla u_t\rangle=s
\]
with the boundary condition $u(0)=u_0, u(1)=u_1$ has a unique smooth solution $u^s$ with uniform $C^{1, 1}$ estimates for $0<s\leq 1$. When $s\rightarrow 0$, $u^s$ converges uniformly in $C^{1, \alpha}$ to the unique $C^{1, 1}$ solution of the geodesic equation
\[
u_{tt}\sigma_k(A_u)-\langle T_{k-1}(A_u),  \nabla u_t\otimes \nabla u_t\rangle=0.
\]
With this approximation, we have the following,
\begin{lemma}
The functional $F$ is convex along the $C^{1, 1}$ geodesic and $F$ takes its minimum in $\cC_k^+$ at any smooth critical point with $\sigma_k(A_u)=const$.
\end{lemma}

Now suppose $u_0$ and $u_1$ are both smooth critical points of $F$ in $\cC_k^+$. Connecting $u_0$ and $u_1$ by the unique $C^{1, 1}$ geodesic, then we have the following
\begin{lemma}
Then either $u(t)=u_0+at$ or $(M^n, g_{u(t)})$ is conformally diffeomorphic to the round sphere $S^n$.
\end{lemma}

As a corollary, this proved the following
\begin{thm}\label{unique01}
Suppose $(M^{2k}, g)$ is locally conformally flat such that $g\in \Gamma_k^+$. Then any solution of $\sigma_k(A_u)=const$ in $\cC_k^+$ is conformally diffeomorphic to the round sphere $S^{2k}$.
\end{thm}

The proof of Theorem \ref{unique01} relies on the convexity of $F$ along the $C^{1, 1}$ geodesic. The proof is rather standard by an approximation argument and the details are almost identical to the case $n=4=2k$, as in \cite[Appendix]{He18}. We skip the details.

We shall mention that A.B. Li and Y. Y. Li \cite{LL03, LL05} proved the conformal invariance of a large class of fully nonlinear equations on $S^n$ by studying the Liouville type theorem.
Their classical results  include Theorem \ref{unique01} as a special case.

\end{document}